\documentclass[12pt,pdftex,a4paper]{amsart}

\selectfont

\usepackage{wrapfig}

\usepackage{qtree}
\usepackage{forest}

\usepackage{caption}

\usepackage{etoolbox}
\usepackage{epic,eepic,ecltree}
\usepackage{graphicx}
\usepackage{tikz}
\usetikzlibrary{decorations.text}
\usepackage{amssymb,color, euscript, enumerate}
\usepackage{amsthm}
\usepackage{amsmath}
\usepackage{braket}
\usepackage{amscd}
\usepackage{txfonts}
\usepackage{comment}
\usepackage{bm}
\usepackage{amscd}
\numberwithin{equation}{section}

\makeatletter
\let\old@tocline\@tocline
\let\section@tocline\@tocline
\newcommand{\subsection@dotsep}{4.5}
\newcommand{\subsubsection@dotsep}{4.5}
\patchcmd{\@tocline}
  {\hfil}
  {\nobreak
     \leaders\hbox{$\m@th
        \mkern \subsection@dotsep mu\hbox{.}\mkern \subsection@dotsep mu$}\hfill
     \nobreak}{}{}
\let\subsection@tocline\@tocline
\let\@tocline\old@tocline

\patchcmd{\@tocline}
  {\hfil}
  {\nobreak
     \leaders\hbox{$\m@th
        \mkern \subsubsection@dotsep mu\hbox{.}\mkern \subsubsection@dotsep mu$}\hfill
     \nobreak}{}{}
\let\subsubsection@tocline\@tocline
\let\@tocline\old@tocline

\let\old@l@subsection\l@subsection
\let\old@l@subsubsection\l@subsubsection

\def\@tocwriteb#1#2#3{%
  \begingroup
    \@xp\def\csname #2@tocline\endcsname##1##2##3##4##5##6{%
      \ifnum##1>\c@tocdepth
      \else \sbox\z@{##5\let\indentlabel\@tochangmeasure##6}\fi}%
    \csname l@#2\endcsname{#1{\csname#2name\endcsname}{\@secnumber}{}}%
  \endgroup
  \addcontentsline{toc}{#2}%
    {\protect#1{\csname#2name\endcsname}{\@secnumber}{#3}}}%

\newlength{\@tocsectionindent}
\newlength{\@tocsubsectionindent}
\newlength{\@tocsubsubsectionindent}
\newlength{\@tocsectionnumwidth}
\newlength{\@tocsubsectionnumwidth}
\newlength{\@tocsubsubsectionnumwidth}
\newcommand{\settocsectionnumwidth}[1]{\setlength{\@tocsectionnumwidth}{#1}}
\newcommand{\settocsubsectionnumwidth}[1]{\setlength{\@tocsubsectionnumwidth}{#1}}
\newcommand{\settocsubsubsectionnumwidth}[1]{\setlength{\@tocsubsubsectionnumwidth}{#1}}
\newcommand{\settocsectionindent}[1]{\setlength{\@tocsectionindent}{#1}}
\newcommand{\settocsubsectionindent}[1]{\setlength{\@tocsubsectionindent}{#1}}
\newcommand{\settocsubsubsectionindent}[1]{\setlength{\@tocsubsubsectionindent}{#1}}

\renewcommand{\l@section}{\section@tocline{1}{\@tocsectionvskip}{\@tocsectionindent}{}{\@tocsectionformat}}%
\renewcommand{\l@subsection}{\subsection@tocline{2}{\@tocsubsectionvskip}{\@tocsubsectionindent}{}{\@tocsubsectionformat}}%
\renewcommand{\l@subsubsection}{\subsubsection@tocline{3}{\@tocsubsubsectionvskip}{\@tocsubsubsectionindent}{}{\@tocsubsubsectionformat}}%
\newcommand{\@tocsectionformat}{}
\newcommand{\@tocsubsectionformat}{}
\newcommand{\@tocsubsubsectionformat}{}
\expandafter\def\csname toc@1format\endcsname{\@tocsectionformat}
\expandafter\def\csname toc@2format\endcsname{\@tocsubsectionformat}
\expandafter\def\csname toc@3format\endcsname{\@tocsubsubsectionformat}
\newcommand{\settocsectionformat}[1]{\renewcommand{\@tocsectionformat}{#1}}
\newcommand{\settocsubsectionformat}[1]{\renewcommand{\@tocsubsectionformat}{#1}}
\newcommand{\settocsubsubsectionformat}[1]{\renewcommand{\@tocsubsubsectionformat}{#1}}
\newlength{\@tocsectionvskip}
\newcommand{\settocsectionvskip}[1]{\setlength{\@tocsectionvskip}{#1}}
\newlength{\@tocsubsectionvskip}
\newcommand{\settocsubsectionvskip}[1]{\setlength{\@tocsubsectionvskip}{#1}}
\newlength{\@tocsubsubsectionvskip}
\newcommand{\settocsubsubsectionvskip}[1]{\setlength{\@tocsubsubsectionvskip}{#1}}

\patchcmd{\tocsection}{\indentlabel}{\makebox[\@tocsectionnumwidth][l]}{}{}
\patchcmd{\tocsubsection}{\indentlabel}{\makebox[\@tocsubsectionnumwidth][l]}{}{}
\patchcmd{\tocsubsubsection}{\indentlabel}{\makebox[\@tocsubsubsectionnumwidth][l]}{}{}

\newcommand{\@sectypepnumformat}{}
\renewcommand{\contentsline}[1]{%
  \expandafter\let\expandafter\@sectypepnumformat\csname @toc#1pnumformat\endcsname%
  \csname l@#1\endcsname}
\newcommand{\@tocsectionpnumformat}{}
\newcommand{\@tocsubsectionpnumformat}{}
\newcommand{\@tocsubsubsectionpnumformat}{}
\newcommand{\setsectionpnumformat}[1]{\renewcommand{\@tocsectionpnumformat}{#1}}
\newcommand{\setsubsectionpnumformat}[1]{\renewcommand{\@tocsubsectionpnumformat}{#1}}
\newcommand{\setsubsubsectionpnumformat}[1]{\renewcommand{\@tocsubsubsectionpnumformat}{#1}}
\renewcommand{\@tocpagenum}[1]{%
  \hfill {\mdseries\@sectypepnumformat #1}}

\let\oldappendix\appendix
\renewcommand{\appendix}{%
  \leavevmode\oldappendix%
  \addtocontents{toc}{%
    \protect\settowidth{\protect\@tocsectionnumwidth}{\protect\@tocsectionformat\sectionname\space}%
    \protect\addtolength{\protect\@tocsectionnumwidth}{2em}}%
}
\makeatother

\makeatletter
\settocsectionnumwidth{2em}
\settocsubsectionnumwidth{2.5em}
\settocsubsubsectionnumwidth{3em}
\settocsectionindent{1pc}%
\settocsubsectionindent{\dimexpr\@tocsectionindent+\@tocsectionnumwidth}%
\settocsubsubsectionindent{\dimexpr\@tocsubsectionindent+\@tocsubsectionnumwidth}%
\makeatother

\settocsectionvskip{10pt}
\settocsubsectionvskip{0pt}
\settocsubsubsectionvskip{0pt}
    
\settocsectionformat{\bfseries}
\settocsubsectionformat{\mdseries}
\settocsubsubsectionformat{\mdseries}
\setsectionpnumformat{\bfseries}
\setsubsectionpnumformat{\mdseries}
\setsubsubsectionpnumformat{\mdseries}

\let\oldtableofcontents\tableofcontents
\renewcommand{\tableofcontents}{%
  \vspace*{-\linespacing}
  \oldtableofcontents}

\usetikzlibrary{cd, decorations.pathmorphing}

\newcommand{\comp}{{\mathrm{comp}}}

\newcommand{\Hom}{{\mathrm{Hom}}}

\newcommand{\an}{{\mathrm{an}}}

\newcommand{\cC}{{\mathcal C}}

\newcommand{\TY}{\mathrm{TY}}

\newcommand{\Irr}{{\text{Irr}(V)}}

\newcommand{\Vmodu}{{\underline{V{\text{-mod}}}}}

\newcommand{\Vmodf}{{\underline{V\text{-mod}}_{f.g.}}}

\newcommand{\Lmod}{{{L(\ft,0)\text{-mod}}}}

\newcommand{\Vmodfr}{{\underline{V\text{-mod}}_{f.g.}^r}}

\newcommand{\CPaB}{{\underline{\text{PaB}}}}
\newcommand{\PaB}{{\underline{\text{PaB}}}}

\newcommand{\Z}{\mathbb{Z}}
\newcommand{\R}{\mathbb{R}}
\newcommand{\C}{\mathbb{C}}

\newcommand{\Y}{\mathcal{Y}}

\newcommand{\Xr}{{X_r(\C)}}

\newcommand{\Or}{\mathrm{O}_{\Xr}}

\newcommand{\Mr}{{M_{[0;r]}}}

\newcommand{\mr}{{m_{[0;r]}}}

\newcommand{\Dr}{{\mathrm{D}_{\Xr}}}

\newcommand{\Leaf}{\text{Leaf}}

\newcommand{\alg}{\text{alg}}

\newcommand{\conv}{\text{conv}}

\newcommand{\va}{\bm{1}}
\newcommand{\vac}{\bm{1}}
\newcommand{\1}{\bm{1}}

\newcommand{\id}{{\mathrm{id}}}

\newcommand{\z}{{\bar{z}}}

\newcommand{\pa}{{\partial}}

\newcommand{\al}{\alpha}

\newcommand{\ga}{\gamma}

\newcommand{\ze}{\zeta}

\newcommand{\D}{\bar{D}}
\newcommand{\om}{\omega}
\newcommand{\la}{\lambda}

\newcommand{\si}{\sigma}

\newcommand{\Log}{\mathrm{Log}}
\newcommand{\Arg}{\mathrm{Arg}}

\newcommand{\Is}{{\mathrm{IS}}}

\newcommand{\ft}{\frac{1}{2}}

\newcommand{\fs}{\frac{1}{16}}

\newcommand{\CB}{{\mathcal{CB}}}

\newcommand{\End}{\mathrm{End}}

\newcommand{\cat}{{\underline{\mathrm{Cat}}_\mathbb{C}}}

\newcommand{\Tr}{{\mathcal{T}}}

\newcommand{\cut}{\mathrm{cut}}

\newcommand{\op}{{\mathrm{op}}}

\makeatletter
\NewDocumentCommand{\extp}{e{^}}{%
  \mathop{\mathpalette\extp@{#1}}\nolimits
}
\NewDocumentCommand{\extp@}{mm}{%
  \bigwedge\nolimits\IfValueT{#2}{^{\extp@@{#1}#2}}%
  \IfValueT{#1}{\kern-2\scriptspace\nonscript\kern2\scriptspace}%
}
\newcommand{\extp@@}[1]{%
  \mkern
    \ifx#1\displaystyle-1.8\else
    \ifx#1\textstyle-1\else
    \ifx#1\scriptstyle-1\else
    -0.5\fi\fi\fi
  \thinmuskip
}
\makeatletter

\newtheorem{thm}{Theorem}[section]
\newtheorem{dfn}[thm]{Definition}
\newtheorem{lem}[thm]{Lemma}
\newtheorem{prop}[thm]{Proposition}
\newtheorem{cor}[thm]{Corollary}
\newtheorem{rem}[thm]{Remark}

\newtheorem{mainthm}{Main Theorem}

\begin{document}

\begin{center}
{{\LARGE \bf 
Conformal blocks, parenthesized braid operad, and $c=1/2$ Virasoro vertex operator algebra}
} \par \bigskip

\renewcommand*{\thefootnote}{\fnsymbol{footnote}}
{\normalsize
Yuto Moriwaki \footnote{email: \texttt{moriwaki.yuto (at) gmail.com}}
}
\par \bigskip
{\footnotesize 
RIKEN Center for Interdisciplinary Theoretical and
Mathematical Sciences (iTHEMS), RIKEN,\\ Wako 351-0198, Japan}

\par \bigskip
\end{center}

\noindent

\vspace{5mm}

\begin{center}
\textbf{\large Abstract}
\end{center}

We review the construction of a pseudo-braided category structure on the $C_1$-cofinite module category of a vertex operator algebra using conformal blocks and analytic continuation along paths in configuration spaces. In the rational $C_2$-cofinite case, the pseudo-braided category is represented by tensor products and becomes a balanced braided tensor category.
We then compute all four-point conformal blocks of the Virasoro vertex operator algebra of central charge $1/2$ in terms of hypergeometric functions. We explain how analytic continuation of these blocks determines the braiding and associator, and identify the resulting module category with the Tambara--Yamagami category over $\mathbb{Z}_2$ as a balanced braided tensor category.

\section{Introduction}

The purpose of this note is twofold. First, we recall the construction in
\cite{M1} which equips the $C_1$-cofinite module category of a vertex operator algebra with a
unital pseudo-braided category structure by using conformal blocks and analytic continuation along paths in configuration spaces.
Second, following the method of Belavin--Polyakov--Zamolodchikov
\cite{BPZ}, we explicitly compute all four-point conformal blocks for the
Virasoro vertex operator algebra \(L(\frac12,0)\) in terms of hypergeometric
functions. We then explain how the balanced braided tensor category structure
is recovered from analytic continuation of these conformal blocks.

A \emph{vertex operator algebra} is an algebra with infinitely many products
depending on a formal parameter \cite{Bo,FLM,LL}:
\[
  \cdot_z:V\otimes V\longrightarrow V((z)),\qquad
  a\cdot_z b
  =
  Y(a,z)b
  =
  \sum_{n\in\mathbb Z}a(n)b\,z^{-n-1}.
\]
Iterating these products gives parenthesized expressions, for example
\[
  a_1\cdot_{z_1-z_r}
  \bigl(a_2\cdot_{z_2-z_r}(
  \cdots(a_{r-1}\cdot_{z_{r-1}-z_r}a_r)\cdots)\bigr)\quad\text{ or }\quad
  (a_2\cdot_{z_2-z_1}a_1)\cdot_{z_1-z_4}
  (a_3\cdot_{z_3-z_4}a_4).
\]
These expressions are not only formal. They are power series which converge
in certain regions of the configuration space
\[
  X_r(\mathbb C)
  =
  \{(z_1,\ldots,z_r)\in\mathbb C^r
  \mid z_i\neq z_j \text{ for } i\neq j\}.
\]
For instance, the product \((a_1\cdot_{z_1-z_2} a_2)\cdot_{z_2-z_3} a_3\) corresponds to the
region
\[
  U_{(12)3}
  =
  \{(z_1,z_2,z_3)\in X_3(\mathbb C)
  \mid |z_1-z_2|<|z_2-z_3|\}.
\]

Let \(\Tr_r\) be the set of binary trees whose leaves are labeled by
\(\{1,\ldots,r\}\). Each \(A\in\Tr_r\) determines an order and a
parenthesization of an \(r\)-fold product. In \cite{M1}, we associated with
each \(A\in\Tr_r\) a simply connected domain
\[
  U_A\subset X_r(\mathbb C)
\]
on which the corresponding iterated product is convergent. These domains are
defined by local coordinates attached to the vertices and internal edges of
the tree. 

Let \(M_{[0;r]}=(M_0,M_1,\ldots,M_r)\) be a sequence of \(V\)-modules. The
\emph{conformal block} associated with \(M_{[0;r]}\) is the sheaf of holomorphic
solutions of a certain \(D\)-module on \(X_r(\mathbb C)\) \cite{FB,NT,M1}:
\[
  \CB_{M_{[0;r]}}(U)
  =
  \Hom_{D_{X_r(\mathbb C)}}
  \bigl(D_{M_{[0;r]}},\mathcal O_{X_r(\mathbb C)}(U)\bigr).
\]
Here \(M_0\) is contravariant and may be thought of as the module inserted at
infinity. Under the finiteness assumptions used in this note, this sheaf is
locally constant; see \cite{H2} and \cite[Section 4]{M1}. Hence paths in \(X_r(\mathbb C)\) act on conformal blocks
by analytic continuation.

\begin{wrapfigure}{r}{0.4\textwidth}
  \centering
  \vspace{-2mm}
    \includegraphics[width=\linewidth]{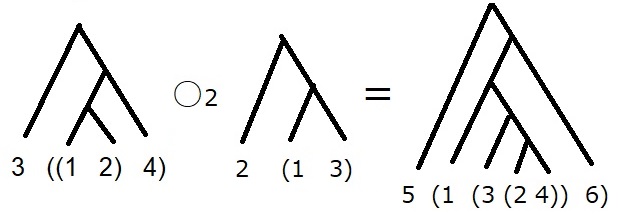}
\caption{composition of trees}
  \label{fig_tree}
\end{wrapfigure}
%
The domains \(U_A\) are chosen so that conformal blocks can be glued in a way
compatible with tree insertion.
If \(A\in\Tr_r\), \(B\in\Tr_s\), and
\(1\leq p\leq r\), then the insertion of \(B\) into the \(p\)-th leaf of \(A\)
gives a tree
\begin{align*}
A\circ_p B\in\Tr_{r+s-1},\qquad\qquad\qquad\qquad
\end{align*}
and there is a corresponding composition map
\[
  \mathrm{glue}_p:
  \CB_{[0;r]}(U_A)\otimes \CB_{[p;s]}(U_B)
  \longrightarrow
  \CB_{[0;r+s-1]}(U_{A\circ_p B}).
\]

\begin{wrapfigure}{r}{0.4\textwidth}
  \centering
  \vspace{-5mm}
    \includegraphics[scale=0.13]{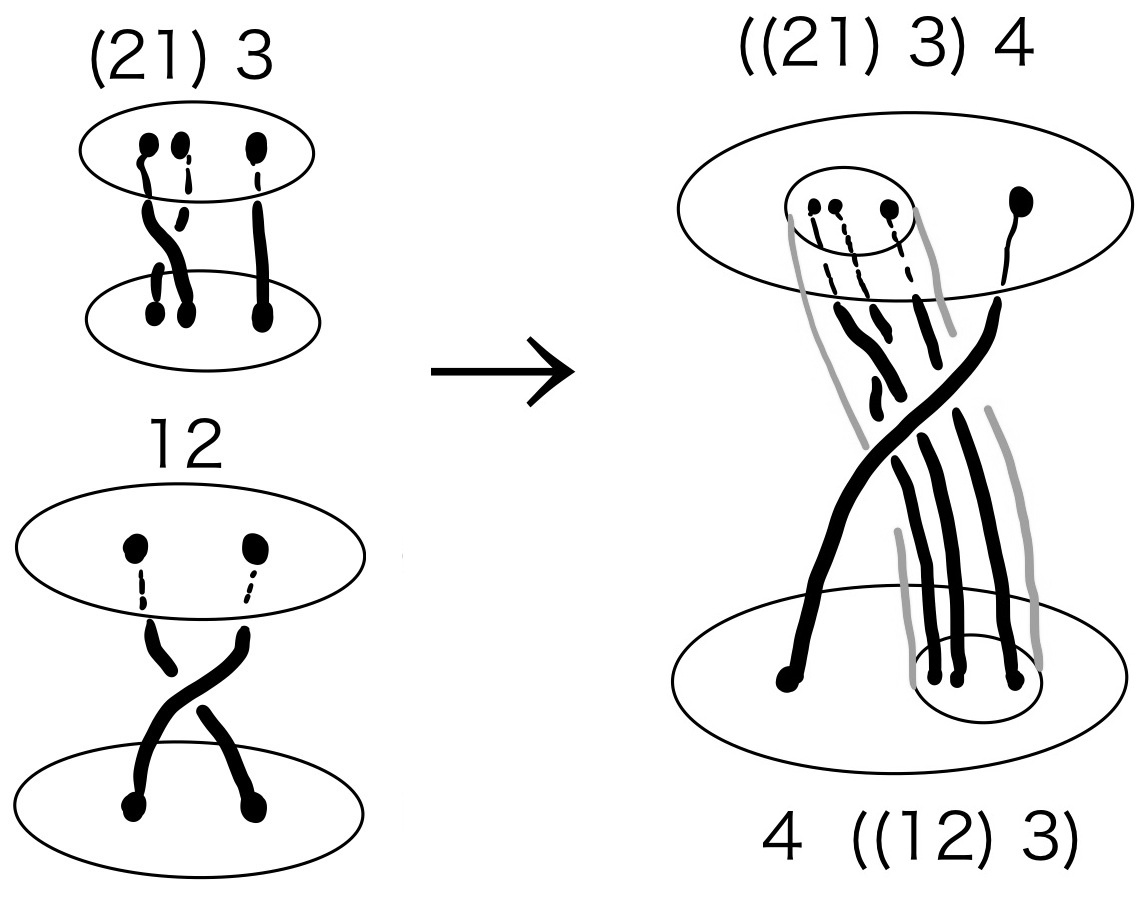}
\caption{composition of paths}
  \label{fig_thin}
\end{wrapfigure}
The \emph{parenthesized braid operad} \(\PaB\) keeps track of both the insertion of
trees and analytic continuation between the corresponding regions.
 Its objects
in arity \(r\) are the trees in \(\Tr_r\), and its morphisms are braids between
the associated parenthesized words. Geometrically, a morphism in \(\PaB(r)\)
is represented by a path in the configuration space \(X_r(\mathbb C)\), and it
acts on conformal blocks by analytic continuation. With these operations, the
spaces
\[
  \Hom_A(M_1,\ldots,M_r;M_0)
  =
  \CB_{M_{[0;r]}}(U_A)
\]
form the multi-hom spaces of a \emph{unital pseudo-braided category} \cite{So,M1}.

Let \(\Vmodf\) denote the category of \(C_1\)-cofinite
\(V\)-modules whose contragredient modules are finitely generated (see
Section~4 for the precise definition).
\begin{mainthm}\cite[Theorem~B.2]{M1}\label{thm_A}
If \(V^\vee\) is a finitely generated \(V\)-module, then
\(V\)-mod\(_{f.g.}\) inherits a unital pseudo-braided category structure.
\end{mainthm}

When \(V\) is rational \(C_2\)-cofinite, the above pseudo-braided category
is represented by tensor products. More precisely, let
\(\{N_\lambda\}_{\lambda\in\Irr}\) be a complete set of representatives
of simple \(V\)-modules. Define a bifunctor 
\begin{align}
\boxtimes:\Vmodf\times \Vmodf \rightarrow \Vmodf \label{intro_bifunctor}
\end{align}
by
\begin{align}
  M_1\boxtimes M_2
  =
  \bigoplus_{\lambda\in\Irr}
  N_\lambda\otimes_{\mathbb C}
  I\binom{N_\lambda}{M_1\,M_2}^{*}.
  \label{def_otimes}
\end{align}
Here \(I\binom{N_\lambda}{M_1\,M_2}^{*}\) denotes the dual vector space of the space of intertwining operators of type \(\binom{N_\lambda}{M_1\,M_2}\).
 The dual appears because \(I\binom{N_\lambda}{M_1\,M_2}\) is contravariant in \(M_1\) and \(M_2\), whereas \(\boxtimes\) in \eqref{intro_bifunctor} is a covariant bifunctor.

For every binary tree \(A\), let \(\boxtimes_A(M_1,\ldots,M_r)\) be the
iterated tensor product of shape \(A\), e.g., $\boxtimes_{(31)2}(M_1,M_2,M_3)= (M_3 \boxtimes M_1)\boxtimes M_2$.
 Then we have:
\begin{mainthm}\cite[Theorems~7.22 and~7.23]{M1}\label{thm_B}
If \(V\) is rational and \(C_2\)-cofinite, then there are natural
isomorphisms
\[
  \Hom_{V\text{-mod}_{f.g.}}
  \bigl(\boxtimes_A(M_1,\ldots,M_r),M_0\bigr)
  \cong
  \CB_{M_{[0;r]}}(U_A)\qquad \text{ for all \(A\in\Tr_r\).}
\]
In particular, \(\Vmodf\) inherits a
balanced braided tensor category structure.
\end{mainthm}

This gives a construction, from the viewpoint of conformal blocks and the
parenthesized braid operad, of a braided tensor category structure on the module
category of a rational \(C_2\)-cofinite vertex operator algebra. For this class
of vertex operator algebras, the existence of a braided tensor category
structure on the module category was established by Huang and Lepowsky
\cite{HL1,HL2,HL3,HL4,H1,H2,H3}. In the present approach, analytic
continuation of conformal blocks along the generators of \(\CPaB\) determines
the braiding and associator via the Yoneda isomorphisms.


\begin{wrapfigure}{r}{0.2\textwidth}
  \centering
  \vspace{-5mm}
    \includegraphics[width=2cm]{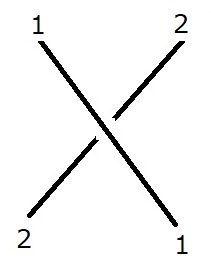}
    \caption{$\si$}\label{fig_sigma}
\end{wrapfigure}
Let
\[
  \sigma:(12)\longrightarrow(21)\qquad\qquad\qquad\qquad\qquad
\]
be the generator of \(\PaB(2)\) shown in Figure~\ref{fig_sigma}. Throughout
this note, we choose its geometric representative in \(X_2(\mathbb C)\) to be the clockwise half-turn path
\footnote{The counterclockwise half-turn gives
the reversed braiding. The clockwise convention is the one compatible with the
standard twist \(\theta_M=\exp(2\pi iL(0))\).} (see Remark \ref{rem_path_explicit}).
For \(M_1,M_2\in V\)-mod\(_{f.g.}\), let
\begin{align}
  R_{\lambda;M_1,M_2}:
  I\binom{N_\lambda}{M_1\,M_2}
  \longrightarrow
  I\binom{N_\lambda}{M_2\,M_1}\qquad\qquad\qquad\qquad\qquad
  \label{intro_R}
\end{align}
be the linear isomorphism obtained by analytic continuation along this path,
after identifying intertwining operators with conformal blocks on $X_2(\C)$.
Under the
decomposition \eqref{def_otimes}, the categorical braiding $B_{M_1,M_2}:M_1\boxtimes M_2 \rightarrow M_2 \boxtimes M_1$ is given by
\[
  B_{M_1,M_2}
  =
  \bigoplus_{\lambda\in\Irr}
  \id_{N_\lambda}\otimes
  \left((R_{\lambda;M_1,M_2})^{-1}\right)^*,
\]
where $\left((R_{\lambda;M_1,M_2})^{-1}\right)^*$ is the inverse transpose of \eqref{intro_R}.

Similarly, the associator is described by the analytic continuation between
\[
  \mathcal H^{(12)3}_\lambda
  =
  \bigoplus_{\mu\in\Irr}
  I\binom{N_\lambda}{N_\mu\,M_3}
  \otimes
  I\binom{N_\mu}{M_1\,M_2},
\quad
  \mathcal H^{1(23)}_\lambda
  =
  \bigoplus_{\nu\in\Irr}
  I\binom{N_\lambda}{M_1\,N_\nu}
  \otimes
  I\binom{N_\nu}{M_2\,M_3}.
\]
\begin{wrapfigure}{r}{0.2\textwidth}
  \centering
  \vspace{-3mm}
    \includegraphics[width=2.5cm]{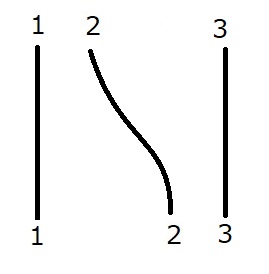}
    \caption{$\al$}\label{fig_alpha}
\end{wrapfigure}
These spaces are naturally identified with
$\CB_{N_\lambda;M_1,M_2,M_3}(U_{(12)3})$ and $\CB_{N_\lambda;M_1,M_2,M_3}(U_{1(23)})$ by Theorem \ref{thm_B}.
Let $F_{\lambda;M_1,M_2,M_3}:
  \mathcal H^{(12)3}_\lambda
  \to
  \mathcal H^{1(23)}_\lambda$
be the linear isomorphism obtained by analytic continuation along 
\(\alpha:(12)3\to1(23)\) in $X_3(\C)$ shown in Fig. \ref{fig_alpha}. Then the categorical associator is given by
\[
  \alpha_{M_1,M_2,M_3}
  =
  \bigoplus_{\lambda\in\Irr}
  \id_{N_\lambda}\otimes
  (F_{\lambda;M_1,M_2,M_3}^{-1})^*.
\]

The final section is devoted to the basic example \(V=L(\frac12,0)\), the
Virasoro vertex operator algebra of central charge \(c=\frac12\). This VOA
has three irreducible modules
\[
  L\left(\frac12,0\right),\qquad
  L\left(\frac12,\frac12\right),\qquad
  L\left(\frac12,\frac1{16}\right).
\]
The relevant four-point conformal blocks are solutions of BPZ-type
differential equations \cite{BPZ}. Although these functions are classical in conformal
field theory, 
we list them explicitly in Table \ref{table_block_intro} and determine the categorical braiding and
associator from the inverse transposes of the corresponding analytic
continuations.
\begin{table}[h]
\caption{Virasoro conformal blocks}
\label{table_block_intro}
  \begin{tabular}{|l|c||c|} \hline
$(h_0,h_1,h_2,h_3)$ & $h$ & $C_{h_0,h_1,h_2,h_3}^{h}(z_1,z_2,z_3)$ \\ \hline \hline
$(\frac{1}{16},\frac{1}{16},\frac{1}{16},\frac{1}{16}) $
& $0$ & $\frac{1}{2}\{z_{13}z_{23}z_{12}\}^{-\frac{1}{8}}
\Bigl((z_{13}^{\frac{1}{2}}+z_{23}^{\frac{1}{2}})^{\frac{1}{2}} +
(z_{13}^{\frac{1}{2}}-z_{23}^{\frac{1}{2}})^{\frac{1}{2}} \Bigr)$
 \\
    & $\frac{1}{2}$ & $\frac{1}{2} \{z_{13}z_{23}z_{12}\}^{-\frac{1}{8}}
\Bigl((z_{13}^{\frac{1}{2}}+z_{23}^{\frac{1}{2}})^{\frac{1}{2}} - 
(z_{13}^{\frac{1}{2}}-z_{23}^{\frac{1}{2}})^{\frac{1}{2}} \Bigr)$
 \\ \hline
$(\frac{1}{2},\frac{1}{2},\frac{1}{2},\frac{1}{2}) $
& $0$ & $\{z_{13}z_{23}z_{12}\}^{-1}(z_1^2+z_2^2+z_3^2-z_{1}z_2-z_1z_3-z_2z_3)$ \\ \hline

$(\frac{1}{2},\frac{1}{2},\frac{1}{16},\frac{1}{16}) $
& $0$ & $\ft z_{23}^{-\frac{1}{8}}\{z_{13}z_{12}\}^{-\frac{1}{2}}
(2z_1-z_2-z_3)$ \\
$(\frac{1}{2},\frac{1}{16},\frac{1}{2},\frac{1}{16}) $
& $\frac{1}{16}$ &
$\frac{1}{2}z_{13}^{-\frac{1}{8}}\{z_{23}z_{12}\}^{-\frac{1}{2}}
(z_1-2z_2+z_3)$ \\
$(\frac{1}{16},\frac{1}{2},\frac{1}{2},\frac{1}{16}) $
& $\frac{1}{16}$ & $\frac{1}{2}(z_{13}z_{23})^{-\frac{1}{2}} z_{12}^{-1}(z_1+z_2-2z_3)$\\
$(\frac{1}{2},\frac{1}{16},\frac{1}{16},\frac{1}{2}) $
& $\frac{1}{16}$ & $\frac{1}{2}\{z_{13}z_{23}\}^{-\frac{1}{2}}z_{12}^{-\frac{1}{8}}(z_1+z_2-2z_3)$ \\
$(\frac{1}{16},\frac{1}{2},\frac{1}{16},\frac{1}{2}) $
& $\frac{1}{16}$ & $\frac{1}{2}z_{13}^{-1}\{z_{23}z_{12}\}^{-\frac{1}{2}}(z_1+z_3-2z_2)$\\
$(\frac{1}{16},\frac{1}{16},\frac{1}{2},\frac{1}{2}) $
& $0$ & $\ft y^{-1}\{z_{13}z_{12}\}^{-\frac{1}{2}}(2z_1-z_2-z_3)$ \\ \hline
\end{tabular}
\end{table}

The result is that the balanced braided tensor category \(L(\frac12,0)\)-mod
is equivalent to the \emph{Tambara--Yamagami category} \cite{TY}
\[
  \mathrm{TY}(\mathbb Z_2,\chi,2^{-1/2}),
  \qquad
  \chi(1,1)=-1.
\]
Thus the example gives a concrete comparison between the BPZ computation of
Virasoro conformal blocks and the \(F\)- and \(R\)-symbols of the resulting
braided tensor category.

While this note focuses on genus-zero theories, we also mention the
higher-genus theory of conformal blocks developed by Damiolini, Gibney,
Tarasca, and Krashen \cite{DGT1,DGT2}. It would be interesting to clarify how
their algebro-geometric framework is related to the OPE-type convergence
domains and operadic structures considered here; see also \cite{DW},
especially Expectation~5.4.

The paper is organized as follows. Section~2 recalls the domains of convergence
\(U_A\subset X_r(\mathbb C)\). Section~3 reviews the consistency of VOA
products under analytic continuation. Sections~4 and~5 recall intertwining
operators, conformal blocks, expansion maps, and gluing. Section~6 reviews
pseudo-braided categories and explains Main Theorems A and B. Section~7
computes the conformal blocks of \(L(\frac12,0)\), derives the braiding and
associator via inverse transpose of continuation maps, and identifies the
resulting balanced braided tensor category with a Tambara--Yamagami category.

\section{Configuration space and trees}
\label{sec_model_C}

In Section \ref{sec_model_C}, we recall the definition of the open regions $U_A,{U}_A^c \subset X_r(\C)$ from \cite[Section 3]{M1}.

Given a product (a binary operation), repeating them creates $n$-ary operations. 
In general, $n$-ary operations depend on orders and parentheses of binary operations,
and all possible $n$-ary operations have a one-to-one correspondence with binary trees.

Let $\Tr_{r}$ be the set of all binary trees whose leaves are labeled by $[r]=\{1,2,\dots,r\}$.
Each element in $\Tr_r$ can be regarded as a parenthesized word of $\{1,2,\dots,r\}$,
that is, non-associative, non-commutative monomials on this set in which every letter appears exactly once.
For example, $(5(23))((17)(64))\in T_7$ corresponds to the tree in Fig. \ref{fig_tree_example0}.
Note that $\Tr_0$ consists of the empty word
and $\Tr_3$, for example, is a set of 12 elements
\begin{align*}
\Tr_0&=\{\emptyset\},\\
\Tr_3&=\{1(23), (12)3, 1(32), (13)2, 2(13),(21)3,2(31),(23)1,3(12),(31)2,3(21),(32)1\}
\end{align*}
and
$\Tr_4$ consists of all permutations of $5$ elements
\begin{align*}
\{(12)(34),1(2(34)),1((23)4),(1(23))4,((12)3)4\}.
\end{align*}

\begin{minipage}[c]{.35\textwidth}
\centering
\begin{forest}
for tree={
  l sep=20pt,
  parent anchor=south,
  align=center
}
[
[[5][[2][3]]]
[[[1][7]]
[[6][4]]]
]
\end{forest}
\captionof{figure}{}
\label{fig_tree_example0}
\end{minipage}
\begin{minipage}[l]{.7\textwidth}
\centering
For $A\in \Tr_r$, we will use the following notations:
\begin{align*}
\Leaf(A)&=\{\text{the set of all leaves of }A\}\\
V(A)&=\{\text{the set of all vertices of }A  \text{ which are}\\
&\quad\quad\text{ not leaves}\}\\
E(A)&=\{\text{the set of all edges of }A \text{ which are}\\
&\quad\quad\text{ not connected to leaves}\}.
\end{align*}
\end{minipage}

Let $(z_1,\dots,z_r)$ be the standard coordinates of $\C^r$.
We will define local coordinates associated with trees $\Tr_r$.
Let $A\in \Tr_r$.
For each edge $e\in E(A)$, let $u(e)$ denote the upper vertex and $d(e)$ denote the lower vertex.
Define maps $L,R: V(A)\rightarrow \Leaf(A)$ as follows:
For each vertex $v\in V(A)$, $R(v)$ is defined by the rightmost leaf that is a descendant of $v$
and $L(v)$ by the rightmost leaf among the leaves that are descendants of the left child of $v$. 
Let $t_A$ be the uppermost vertex and $r_A$ be the rightmost leaf among all leaves.
Then, $r_A=R(t_A)$.

\begin{minipage}[c]{.5\textwidth}
\centering
\begin{forest}
for tree={
  l sep=20pt,
  parent anchor=south,
  align=center
}
[$t_{A}$
[$v_1$[5][[2][3]]]
[$v_2$,edge label={node[midway,right]{$e_0$}}[[1][7]]
[[6][4]]]
]
\end{forest}
\end{minipage}
\begin{minipage}[c]{.5\textwidth}
\centering
In the case of the left figure,
\begin{align*}
A&=(5(23))((17)(64)),\\
d(e_0)&=v_2,\quad u(e_0)=t_{A},\\
L(v_1)&=5, \quad R(v_1)=3,\\
L(v_2)&=7, \quad R(v_2)=4,\\
r_A&=4.
\end{align*}
\end{minipage}


The functions $\{z_v:\Xr\rightarrow \C\}_{v\in V(A)}$ and $\{\zeta_e:\Xr\rightarrow \C\}_{e\in E(A)}$ are
defined by
\begin{align}
z_v &= z_{L(v)}-z_{R(v)},\label{eq_vertex_change}\\
\zeta_e &= \frac{z_{d(e)}}{z_{u(e)}}.
\end{align}
This gives a family of $r-1$ functions $\{z_v:\Xr\rightarrow \C\}_{v\in V(A)}$
and a family of $r-2$ functions $\{\zeta_e:\Xr\rightarrow \C\}_{e \in E(A)}$.
Let $z_A:\Xr \rightarrow \C,\quad(z_1,\dots,z_r)\mapsto z_{r_A}$ be the projection onto the $r_A$-th component.
Then,
\begin{align*}
(z_v)_{v\in V(A)}\times z_A: \Xr\rightarrow \C^{r-1}\times \C
\end{align*}
forms local coordinates on $X_r(\C)$.
{\bf The $A$-coordinate system} is the system of functions
\begin{align}
\Psi_A = 
z_A\times x_A\times (\zeta_e)_{e\in E(A)}: \Xr\rightarrow \C\times \C\times  \C^{r-2},
\label{eq_zeta_coordinate}
\end{align}
where $x_A=z_{t_A}:\Xr \rightarrow \C$.

For example, for $A=(23)((15)4) \in \Tr_5$, the $A$-coordinate system is given by:
\begin{align}
\Psi_{(23)((15)4)}=
(z_4,z_3-z_4,\ze_a=\frac{z_2-z_3}{z_3-z_4},\ze_b=\frac{z_1-z_5}{z_5-z_4},\ze_c=\frac{z_5-z_4}{z_3-z_4}),
\label{eq_example_psi}
\end{align}
\begin{minipage}[c]{.5\textwidth}
\centering
where the labels $\{a,b,c\}$ of the edges are given as in Fig. \ref{fig_tree_example2}.
In Fig. \ref{fig_tree_example2}, we have assigned to each vertex $v$ of the tree the coordinate $z_v = z_{L(v)} - z_{R(v)}$. Such a diagram is useful for explicitly describing the $A$-coordinates.
It is easy to see that the inverse function $\Psi_A^{-1}:\C^2\times \C^{E(A)}\rightarrow \C^r$ is a polynomial
of $\{\ze_e\}_{e\in E(A)}$ and $x_A,z_A$.
\end{minipage}
\begin{minipage}[c]{.5\textwidth}
\centering
\begin{forest}
for tree={
  l sep=20pt,
  parent anchor=south,
  align=center
}
[$z_3-z_4$
[$z_2-z_3$,edge label={node[midway,left]{a}}[2][3]]
[$z_5-z_4$,edge label={node[midway,right]{c}}[$z_1-z_5$,edge label={node[midway,left]{b}}[1][5]]
[4]]
]
\end{forest}
\captionof{figure}{}
\label{fig_tree_example2}
\end{minipage}
For example,
\begin{align*}
\Psi_{(23)((15)4)}^{-1}=(z_1,z_2,z_3,z_4,z_5)
=(x_A\ze_c(1+\ze_b)+z_A, (1+\ze_a)x_A+z_A,x_A+z_A,z_A,\ze_cx_A+z_A).
\end{align*}
Thus, we have:
\begin{prop}
\label{prop_psiA}
For any $A\in \Tr_r$, 
$\Psi_A$ is a biholomorphic function from $X_r(\C)$ onto the image in $\C^r$.
Furthermore, $\Psi_A^{-1}$ is a polynomial of $\{\ze_e\}_{e\in E(A)}$ and $x_A,z_A$,
and thus can be extended to a holomorphic function $\Psi_A^{-1}:\C^2\times \C^{E(A)}\rightarrow \C^r$.
\end{prop}

Set 
\begin{align*}
\Or^\alg = \C[z_1,\dots,z_r,(z_i-z_j)^\pm],
\end{align*}
a ring of regular functions on $X_r(\C)$, and
\begin{align*}
T_A &= \C[[\zeta_e\mid e \in E(A)]][z_A,\log x_A, x_A^\C,\log\zeta_e,\zeta_e^\C \mid e \in E(A)],
\end{align*}
which is a space of formal power series spanned by formal power series of the form
\begin{align*}
z_A^n x_A^r (\log x_A)^k \Pi_{e\in E(A)}\zeta_e^{r_e}(\log \zeta_e)^{k_e}F
\end{align*}
with $F\in\C[[\zeta_e\mid e \in E(A)]]$,
$n,k,k_e \in \Z_{\geq 0}$ and $r_e,r \in \C$ ($e\in E(A)$).

Any function of $\Or^\alg$ can be expanded as a formal power series in $T_A$.
For example, in the case of $A=(23)((15)4) \in \Tr_5$, we have:
\begin{align}
\begin{split}
(z_2-z_1)^{-1} &= \left((z_2-z_3)+(z_3-z_4)-(z_5-z_4)-(z_1-z_5)
\right)^{-1}\\
&= (z_3-z_4)^{-1}\left(1 + \frac{(z_2-z_3)}{(z_3-z_4)}-\frac{(z_5-z_4)}{(z_3-z_4)}-\frac{(z_1-z_5)}{(z_3-z_4)}
\right)^{-1}\\
&=x_{(23)((15)4)}^{-1}(1+\zeta_a-\zeta_c-\zeta_b\zeta_c)^{-1}\\
&=x_{(23)((15)4)}^{-1}\sum_{l=0}^\infty (-\zeta_a+\zeta_c+\zeta_b\ze_c)^l
\in \C[[\zeta_a,\zeta_b,\zeta_c]][x_{(23)((15)4)}^{-1}].\label{eq_example_conv}
\end{split}
\end{align}

The series in $T_A$ is called {\it a parenthesized formal power series}.
It is noteworthy that $T_A$ naturally inherits a $\Or^\alg$-algebra structure, by the $\C$-algebra homomorphism:
\begin{align}
e_A:\Or^\alg \rightarrow T_A. \label{lem_module_ort}
\end{align}
In particular, $T_A$ is an $\Or^\alg$-module.

Next, consider the radius of convergence of parenthesized formal power series.
For $p>0$, set
\begin{align*}
\D_p &= \{\ze\in \C\mid |\ze|<p\},\\
\D_p^\times &=\{\ze \in \C\mid 0<|\ze|<p\}.
\end{align*}

Let $\mathfrak{p}=(p_e)_{e\in E(A)}\in \R_{>0}^{E(A)}$.
Let $\C[[\zeta_e\mid e\in E(A)]]_{\mathfrak{p}}^\conv$
be a subspace of $\C[[\zeta_e\mid e\in E(A)]]$
consisting of formal power series which are absolutely convergent in $\Pi_{e\in E(A)}\D_{p_e}$
and set
\begin{align*}
T_A^{\mathfrak{p}}=\C[[\zeta_e\mid e\in E(A)]]_{\mathfrak{p}}^\conv[z_A,z_v^\C,\log z_v\mid v\in V(A)],
\end{align*}
a subspace of $T_A$.
It is important to note that the region of absolute convergence of 
$e_{(23)((15)4)}((z_2-z_1)^{-1})$ in the example  \eqref{eq_example_conv}
is not $|\zeta_a|<1,|\zeta_b|<1,|\zeta_c|<1$.
Since 
$e_{(23)((15)4)}((z_2-z_1)^{-1})=x_{{(23)((15)4)}}^{-1}\sum_{l=0}^\infty (-\zeta_a+\zeta_c+\zeta_b\zeta_c)^l$,
$e_{(23)((15)4)}\left((z_2-z_1)^{-1}\right) \in T_{(23)((15)4)}^{\mathfrak{p}}$
if and only if $p_a+p_b+p_cp_b <1$.


\begin{dfn}
A sequence of positive real numbers $(p_e)_{e\in E(A)} \in \R_{>0}^{E(A)}$ is called {\it $A$-admissible} if $\Psi_A^{-1}(\C\times \C^\times \times \Pi_{e\in E(A)}\D_{p_e}^\times) \subset X_r(\C)$,
where $\Psi_A^{-1}$ is the polynomials in Proposition \ref{prop_psiA}.
\end{dfn}

A convergent series $f\in T_A^{\mathfrak{p}}$ is a multi-valued holomorphic function on $\Psi_A^{-1}(\C \times \C^\times \times \Pi_{e\in E(A)}\D_{p_e}^\times)$ 
because it contains $\log(z_v)$ and $z_v^r$.
Below, we will fix the branch.
For $A$-admissible numbers $\mathfrak{p}$,
set
\begin{align*}
U_{A}^{\mathfrak{p}}&=\Phi_A^{-1}(\C \times \C^\cut \times \Pi_{e\in E(A)}\D_{p_e}^\cut),\\
{U}_{A}^{c,\mathfrak{p}}&=\Phi_A^{-1}(\C \times \C^\times \times \Pi_{e\in E(A)}\D_{p_e}^\times),
\end{align*}
where 
\begin{align*}
\R_-&= \{r\in \R\mid r\leq 0\},\\
\C^\cut &= \C \setminus \R_{-},\\
\D_p^\cut &=\{\ze \in \C^\cut \mid |\ze|<p \}.
\end{align*}

Define the branch of $\Log:\C^\cut \rightarrow \C$ by
\begin{align}
\Log(\exp(\pi i t))= \pi i t
\label{eq_log_def}
\end{align}
for $t\in (-1,1)$. In particular, $\Arg= \mathrm{Im}\,\Log$ takes values in $(-\pi,\pi)$.

Then, each formal power series in $T_A^{\mathfrak{p}}$
can be regarded as a single-valued holomorphic function on
$U_{A}^\mathfrak{p}$.
Set
\begin{align}
U_A  &= \cup_{\mathfrak{p}:A\text{-admissible}}U_{A}^{\mathfrak{p}} \subset \Xr,\\
{U}_{A}^c  &= \cup_{\mathfrak{p}:A\text{-admissible}}{U}_{A}^{c,\mathfrak{p}} \subset \Xr.\label{eq_no_cut}
\end{align}
\begin{lem}
For any $A \in \Tr_{[r]}$, $U_A$ is a connected simply-connected open subset of $\Xr$.
\end{lem}
\begin{proof}
We note that $U_A^{\mathfrak{p}}$ is contractible. For any $A$-admissible $\mathfrak{p}_1, \mathfrak{p}_2 \in \R_{>0}^{E(A)}$, set $q_e= \min \{(p_1)_e,(p_2)_e\}$ for $e\in E(A)$. Then, $\mathfrak{q} =(q_e)_{e\in E(A)}$ is also $A$-admissible.
Hence, $U_A^\mathfrak{q} \subset U_A^{\mathfrak{p}_1} \cap U_A^{\mathfrak{p}_2}$ is non-empty, and thus, $U_A$ is path-connected. Let $\ga:S^1 \rightarrow U_A$ be a continuous map. Since the image of $S^1$ is compact, there are finitely many $A$-admissible $\mathfrak{p}_1,\dots,\mathfrak{p}_N$ such that $\ga:S^1 \rightarrow \bigcup_{i=1}^N U_A^{\mathfrak{p}_i} \subset U_A$. Since $\bigcap_{i=1}^N U_A^{\mathfrak{p}_i}$ is non-empty, by Seifert-van Kampen theorem, $\pi_1(\bigcup_{i=1}^N U_A^{\mathfrak{p}_i})=1$.
\end{proof}

Set
\begin{align*}
T_A^\conv = \cap_{\mathfrak{p}:A\text{-admissible}} T_A^\mathfrak{p} \subset T_A,
\end{align*}
which is a vector space of convergent parenthesized formal power series.
Then, by \eqref{lem_module_ort}, we have a $\C$-algebra homomorphism
\begin{align}
e_A: \Or^\alg \rightarrow T_A^\conv.
\label{eq_eA_conv}
\end{align}

Set
\begin{align*}
\pa_i=\frac{d}{dz_i},
\end{align*}
the partial differential operator  on $\Xr$ with respect to the standard coordinate $(z_1,\dots,z_r)$,
and
\begin{align*}
\Dr = \C[\pa_1,\dots,\pa_r, z_1,\dots,z_r,(z_i-z_j)^\pm\mid 1\leq i <j\leq r],
\end{align*}
a ring of differential operators on $\Xr$. Then, it is clear that $T_A^\conv$ and $\Or^\alg$ are $\Dr$-modules
and $e_A$ in \eqref{eq_eA_conv} is a $\Dr$-module homomorphism.

\section{Consistency of vertex operator algebras}
\label{sec_consistency}
Let $V$ be a vertex operator algebra. Then, in particular, the following properties hold (see for example \cite{LL,FLM,FB}):
\begin{enumerate}
\item[P1)]
$[L(-1),Y(a,z)]=\frac{d}{dz}Y(a,z) = Y(L(-1)a,z)$ for any $a \in V$.
\item[P2)]
$[L(0),Y(a,z)]=(z\frac{d}{dz}+n)Y(a,z)$ for any $a \in V_n$, or equivalently $a(k)$ maps $V_m$ to $V_{m+n-k-1}$ for any $n,m,k \in \Z$.
\item[P3)]
For any $a, b \in V$ and $n\in \Z$,
\begin{align*}
[a(n), Y(b,z)]&=\sum_{k\geq 0} \binom{n}{k} Y(a(k)b,z)z^{n-k},\\
Y(a(n)b,z)&=\sum_{k\geq 0}\binom{n}{k}
\left( 
a(n-k)Y(b,z)(-z)^k
- Y(b,z)a(k)(-z)^{n-k}
\right).
\end{align*}
\end{enumerate}
Throughout this paper, we assume that a vertex operator algebra $V$ satisfies
\begin{itemize}
\item
$V= \bigoplus_{n \geq 0}V_n$, $V_0=\C \va$ and $\dim V_n <\infty$,
\end{itemize}
which is often called of CFT type.
Let $A \in \Tr_r$.
As was mentioned in the introduction, the vertex operator is a product depending on $z$, so we can consider the corresponding parenthesized product.
For example, the product corresponding to $1(23), (12)3 \in \Tr_3$ is given by
\begin{align*}
a_1 \cdot_{z_{13}} \Bigl(a_2 \cdot_{z_{23}} a_3 \Bigr)&= Y(a_1,z_{13})Y(a_2,z_{23})a_3\\
\Bigl( a_1 \cdot_{z_{12}} a_2 \Bigr) \cdot_{z_{23}} a_3&= Y(Y(a_1,z_{12})a_2,z_{23})a_3.
\end{align*}
Here \(Y(a,z)b\) corresponds to the configuration in which \(a\) is inserted
at \(z\) and \(b\) is inserted at \(0\). Thus, in an iterated composition, the
variables of the vertex operators must be chosen according to the shape of
the tree. Each internal vertex \(v\in V(A)\) gives one vertex operator
product. For \(A\in\Tr_r\) and \(a_1,\ldots,a_r\in V\), we denote by
\[
  Y_A(a_1,\ldots,a_r,z_1,\ldots,z_r)
\]
the composition of vertex operators determined by \(A\), with the change of
variables in \eqref{eq_vertex_change}. We call it a {\it parenthesized vertex
operator}.
Set
\begin{align*}
V^\vee = \bigoplus_{n\geq 0}V_n^*,
\end{align*}
where $V_n^*$ is the dual vector space of $V_n$.
For $u\in V^\vee$, $\langle u, \exp(L(-1)z_A) Y_A(a_1,\dots,a_r,z_1,\dots,z_r)\rangle$ is a formal power series in $T_A$ by (P2).

%
%

\begin{dfn}\cite{M2}\label{dfn_VOA_consistent}
A vertex operator algebra $V$ is called consistent if the following properties hold for any $u \in V^\vee$ and $a_{[r]}\in V^{\otimes r}$ for $r\geq 2$:
\begin{enumerate}
\item
For any tree $A\in \Tr_r$, the formal power series $\langle u, \exp(L(-1)z_A)Y_A(a_1,\dots,a_r,z_{[r]})\rangle$ is in $T_A^\conv$, i.e., is absolutely convergent in ${U}_A^c$.
Denote this holomorphic function on ${U}_A^c$ by $S_A(u,a_{[r]},z_{[r]})$.
\item
There exists a sequence of linear maps
\begin{align*}
S_r:V^\vee \otimes V^{\otimes r} \rightarrow \C[z_i,(z_i-z_j)^\pm\mid i\neq j],\quad\quad r \geq 2.
\end{align*}
such that $S_A(u,a_{[r]},z_{[r]}) = S_r |_{{U}_A^c}$ for any tree $A\in \Tr_r$ as holomorphic functions.
\end{enumerate}
\end{dfn}

The following proposition is proved in \cite[Theorem 3.9]{M2}:
\begin{prop}\label{prop_VOA_consistent}
A vertex operator algebra of CFT type is consistent.
\end{prop}

\begin{rem}
In fact, in \cite[Theorem 3.9]{M2}, 
we prove an analogous consistency statement for full vertex operator algebras
satisfying certain assumptions. In that setting the vertex operator is real
analytic:
\begin{align*}
Y(a,z,\z) = \sum_{r,s\in\R}a(r,s)z^{-r-1}\z^{-s-1}.
\end{align*}
\end{rem}

The following properties of $S_A$ motivate the definition of conformal blocks:
\begin{lem}\label{lem_recursive}
Let $A\in \Tr_r$, $a_1,\dots,a_r \in V$ and $u\in V^\vee$. Then, the following properties hold:
\begin{enumerate}
\item
For any $i \in [r]=\{1,\dots,r\}$,
\begin{align*}
S_A(u,L(-1)_i a_{[r]},z_{[r]})
=\frac{d}{dz_i}S_A(u,a_{[r]},z_{[r]}),
\end{align*}
where $L(-1)_i$ is the action of $L(-1)$ on the $i$-th component;
\item
For any $b\in V$ and $n\in\Z$,
\begin{align*}
&S_A(u, b(n)_i a_{[r]},z_{[r]})\\
&= \sum_{j \neq i} \sum_{k \geq 0} \binom{n}{k} e_A((z_j-z_i)^{n-k}) S_A(u, b(k)_j a_{[r]},z_{[r]})
+ \sum_{k \geq 0}\binom{n}{k}(-z_i)^k S_A(b(n-k)^*u, a_{[r]},z_{[r]}),
\end{align*}
where $b(k)_j$ is as above and $(b(n-k)^* u)(\bullet) = u(b(n-k) \bullet )$.
\end{enumerate}
\end{lem}
\begin{proof}
(1) clearly follows from (P1). 
We show (2) for a special case. We leave the general case to the reader (see also \cite[Section 1.3]{M1}).
By (P3), we have:
\begin{align*}
&\langle u, Y(b(n)a_1,z_1) Y(a_2,z_2)Y(a_3,z_3)\vac \rangle -
\sum_{k \geq 0} \binom{n}{k}(-z_1)^{n-k} \langle u, b(n-k) Y(a_1,z_1) Y(a_2,z_2)Y(a_3,z_3)\vac\rangle\\
&=\langle u, Y(a_1,z_1) \left(\sum_{k \geq 0}\binom{n}{k} b(k) (-z_1)^{n-k}\right) Y(a_2,z_2)Y(a_3,z_3)\vac \rangle\\
&=\langle u, Y(a_1,z_1) \left[\left(\sum_{k \geq 0}\binom{n}{k} b(k) (-z_1)^{n-k}\right), Y(a_2,z_2)\right]Y(a_3,z_3)\vac \rangle\\
&+
\langle u, Y(a_1,z_1) Y(a_2,z_2)\left[\left(\sum_{k \geq 0}\binom{n}{k} b(k) (-z_1)^{n-k}\right) ,Y(a_3,z_3)\right]\vac \rangle.
\end{align*}
Since by (P3)
\begin{align*}
\left[\left(\sum_{k \geq 0}\binom{n}{k} b(k) (-z)^{n-k}\right), Y(a,w)\right]
&= \sum_{k \geq 0}\sum_{l \geq 0}\binom{n}{k}\binom{k}{l} (-z)^{n-k}w^{k-l} Y(b(l)a,w) \\
&= \sum_{l \geq 0}\binom{n}{l} (w-z)^{n-l}\Bigl|_{|z|>|w|} Y(b(l)a,w),
\end{align*}
the assertion follows.
\end{proof}

\begin{rem}\label{rem_recursive}
The following points are important for Lemma \ref{lem_recursive}:
\begin{itemize}
\item
The coefficients of the recursion relation in (2) 
do not depend on $A \in \Tr_r$ when they are analytically continued;
\item
Only (P3) is used in the proof of Lemma \ref{lem_recursive};
\item
(1) and (2) define a system of differential equations.
\end{itemize}
\end{rem}
The following corollary is a part of the consistency which is easily derived from Lemma \ref{lem_recursive}:
\begin{cor}
Let $A\in \Tr_r$, $a_1,\dots,a_r \in V$ and $u\in V^\vee$. Then, 
$\langle u, Y(a_{1},z_{1})Y(a_{2},{z_2}) \cdots Y(a_{r},{z_r}) \vac \rangle$ is absolutely convergent  in $U_{1(2(\dots(r-1 r)\dots)}=
\{|z_1|>|z_2|>\cdots >|z_r|\}\subset X_r(\C)$.
Moreover, there is a polynomial $f(z_1,\dots,z_r)\in \C[z_i, (z_i-z_j)^\pm]$ such that:
\begin{align*}
\langle u, Y(a_{\si1},z_{\si1})Y(a_{\si 2},{z_{\si 2}}) \cdots Y(a_{\si r},{z_{\si r}}) \vac \rangle
 =e_{\si 1(\si 2(\cdots (\si (r-1) \si r)))}(f(z_1,\dots,z_r)) 
\end{align*}
for any $\si \in S_r$, where $S_r$ is the permutation group.
\end{cor}
\begin{proof}
Applying Lemma \ref{lem_recursive} to $a_1=a_1(-1)\vac$, we can see that\\
$\langle u, Y(a_{\si1},z_{\si1})Y(a_{\si 2},{z_{\si 2}}) \cdots Y(a_{\si r},{z_{\si r}}) \vac \rangle$ is
a finite sum of polynomials in $\C[z_i, (z_i-z_j)^\pm]$ expanded by $e_{\si 1(\si 2(\cdots (\si (r-1) \si r)))}$.
Hence, it is absolutely convergent in ${U}_{\si 1(\si 2(\cdots (\si (r-1) \si r)))}^c$.
By the locality of a vertex algebra (see for example \cite{LL}), there is an integer $N_{ij} \in \Z_{\geq 0}$ such that:
\begin{align*}
(z_i-z_j)^{N_{ij}}[Y(a_i,z_i),Y(a_j,z_j)]=0.
\end{align*}
Hence, the assertion holds.
\end{proof}
In the next section, we describe the recursion relations among intertwining operators, which is directly related to the definition of conformal blocks.

\section{Modules and intertwining operators}\label{sec_module}
Let \(V\) be a vertex operator algebra. Throughout this note, by a
\(V\)-module we mean a \(V\)-module \(M\) satisfying the following conditions:
\begin{enumerate}
\item the action of \(L(0)\) on \(M\) is locally finite;
\item for each \(h\in\mathbb C\), the generalized \(L(0)\)-eigenspace
\[
  M_h=\{m\in M\mid (L(0)-h)^N m=0 \text{ for } N\gg 0\}
\]
is finite-dimensional;
\item there exist finitely many complex numbers
\(\Delta_1,\ldots,\Delta_N\in\mathbb C\) such that
\[
  M=
  \bigoplus_{i=1}^N
  \bigoplus_{n\geq 0}
  M_{\Delta_i+n}.
\]
\end{enumerate}
We denote by $\Vmodu$ the category of such \(V\)-modules.
Let $M$ be a $V$-module. For any $n \in \Z_{>0}$, set 
\begin{align*}
C_n(M)= \{a(-n)m\mid m\in M \text{ and }a \in\bigoplus_{k\geq 1}V_k \}.
\end{align*}

\begin{dfn}
A $V$-module $M$ is called {\it $C_n$-cofinite} if $M/C_n(M)$ is a finite-dimensional vector space \cite{Zh}.
\end{dfn}

Since $(L(-1)a)(-n)=na(-n-1)$ for any $a \in V$ and $n\in \Z_{>0}$,
$C_{n+1}(M) \subset C_{n}(M)$. Hence, if $M$ is $C_{n+1}$-cofinite,
then $M$ is $C_{n}$-cofinite.
Note that any vertex operator algebra of CFT type is $C_1$-cofinite since $a=a(-1)\vac$.
Denote by $\Vmodf$ the full subcategory of $\Vmodu$ consisting of all $C_1$-cofinite V-modules $M$ such that the contragredient module $M^\vee$ is finitely generated.

We will recall the definition of a logarithmic intertwining operator of a vertex operator algebra from \cite{Mil1,Mil2}.
Let $M_0,M_1,M_2$ be $V$-modules.
\begin{dfn}\label{def_int}
{\it A logarithmic intertwining operator} of type $\binom{M_0}{M_1 M_2}$
is a linear map
$$\Y_1(\bullet,z):M_1 \rightarrow \text{Hom} (M_2,M_0)[[z^\C]][\log z],
\; m \mapsto \Y_1(m,z)=\sum_{k \geq 0} \sum_{r \in \C} m(r;k) z^{-r-1}(\log z)^k$$
such that:
\begin{enumerate}
\item[I1)]
For any $m \in M_1$ and $m' \in M_2$,
$\Y_1(m,z)m' \in M_0[[z]][z^\C,\log z]$.
\item[I2)]
$[L(-1),\Y_1(m,z)]=\frac{d}{dz}\Y_1(m,z)$ for any $m \in M_1$.
\item[I3)]
For any $m \in M_1$, $a \in V$ and $n\in \Z$,
\begin{align*}
[a(n), \Y_1(m,z)]&=\sum_{k\geq 0} \binom{n}{k} \Y_1(a(k)m,z)z^{n-k}\\
\Y_1(a(n)m,z)&=\sum_{k\geq 0}\binom{n}{k}
\left( 
a(n-k)\Y_1(m,z)(-z)^k
- \Y_1(m,z)a(k)(-z)^{n-k}
\right).
\end{align*}
\end{enumerate}
\end{dfn}

The space of all logarithmic intertwining operators of type $\binom{M_0}{M_1M_2}$ forms a vector space,
which is denoted by $I_{\log} \binom{M_0}{M_1M_2}$.
If $\Y_1(\bullet,z) \in I_{\log} \binom{M_0}{M_1M_2}$ does not contain any logarithmic term,
i.e., $\Y_1(m,z) \in \mathrm{Hom}(M_2,M_0)[[z]][z^\C]$ for any $m\in M_1$,
then
$\Y_1(\bullet,z)$ is called {\it an intertwining operator} of type $\binom{M_0}{M_1M_2}$ \cite{FHL}.
Denote by $I\binom{M_0}{M_1M_2}$
the space of all intertwining operators of type $\binom{M_0}{M_1M_2}$.
By (I2) and (I3), (P1), (P2) and (P3) hold for logarithmic intertwining operators.

Let $r\in \Z_{>0}$ and $\{M_i\}_{i =0,1,\dots,r}$ $V$-modules.
Set 
\begin{align*}
\Mr=M_0^\vee \otimes M_1 \otimes \cdots \otimes M_r.
\end{align*}
For each $a \in V$, $i =1,\dots,r $ and $n\in \Z$,
define a linear map $a(n)_i:\Mr\rightarrow \Mr$ by
the action of $a(n):M_i \rightarrow M_i$ on the $i$-th component.
On the $0$-th component, 
define $a(n)_0^*: \Mr \rightarrow \Mr$ by $a(n)^*:M_0^\vee \rightarrow M_0^\vee$,
where 
\begin{align*}
(a(n)^*u)(\bullet)=u(a(n)\bullet) \text{ for } u\in M_0^\vee.
\end{align*}

Let $\Y_1 \in I\binom{N}{M_1M_2}$ and $\Y_2 \in I\binom{M_0}{LM_3}$. Then, $\Y_1$ and $\Y_2$ are formally composable of shape $(12)3$ if $L=N$, that is,
\begin{align*}
\langle u, \exp(L(-1)z_3)
\Y_2(\Y_1(m_1,z_{12})m_2,z_{23})m_3 \rangle \in T_{(12)3},
\end{align*}
for any $u\in M_0^\vee$ and $m_i\in M_i$. The module $L=N$ is called an intermediate module.

In general, for \(A\in\Tr_r\), we say that a sequence
\(\{\mathcal Y_i\}_{i=1}^{r-1}\) of logarithmic intertwining operators is
formally \(A\)-composable with out-state \(M_0\) and in-states
\(M_1,\ldots,M_r\) if the intermediate modules match according to the tree
\(A\), and the resulting composition takes values in \(M_0\).
By (P2), one can easily see that
\begin{align*}
C_A(u,a_{[r]},z_{[r]}) = \langle u,\exp( z_A L(-1)) \Y_A(m_1,\dots,m_r) \rangle \in T_A
\end{align*}
for any $m_i \in M_i$ and $u\in M_0^\vee$.
Note that $C_A$ here depends on the logarithmic intertwining operators.
%
Then, by Remark \ref{rem_recursive}, similarly to Lemma \ref{lem_recursive}, we have:
\begin{lem}\label{lem_module_recursive}
Let $A\in \Tr_r$, $m_i \in M_i$ ($i=1,\dots,r$), $u\in M_0^\vee$ and $\Y_i \in \binom{?}{M_i,?}$ be formally $A$-composable logarithmic intertwining operators.
Then, the linear map
\begin{align*}
C_A: M_0^\vee \otimes M_1 \otimes \cdots \otimes M_r \rightarrow T_A
\end{align*}
satisfies
\begin{enumerate}
\item
For any $i \in [r]$,
\begin{align*}
C_A(u,L(-1)_i m_{[r]},z_{[r]})
=\frac{d}{dz_i}C_A(u,m_{[r]},z_{[r]}).
\end{align*}
\item
For any $b\in V$ and $n\in\Z$,
\begin{align*}
&C_A(u, b(n)_i m_{[r]},z_{[r]})\\
&= \sum_{j \neq i} \sum_{k \geq 0} \binom{n}{k} e_A((z_j-z_i)^{n-k}) C_A(u, b(k)_j m_{[r]},z_{[r]})
+ \sum_{k \geq 0}\binom{n}{k}(-z_i)^k C_A(b(n-k)^*u, m_{[r]},z_{[r]}).
\end{align*}
\end{enumerate}
\end{lem}

As noted above, $C_A$ depends on $A$-composable logarithmic intertwining operators, and there can be many linear maps that satisfy Lemma \ref{lem_module_recursive}.
However, they all satisfy (1) and (2). Therefore, a universal logarithmic intertwining operator can be defined by (1) and (2).
This is the conformal block introduced in the next section.

%

\section{Conformal blocks}
\label{sec_chiral_CB}
In this section, we will review the definition of conformal blocks based on \cite{M1}.
A conformal block is a sheaf of holomorphic solutions of a $\Dr$-module,
defined for a sequence of $V$-modules $(M_i)_{i =0,1,\dots,r}$ (see \cite{FB,NT}).
For the definition and general properties of $\Dr$-modules, we refer the reader to \cite{HTT}.

Let $\Or^\alg \otimes \Mr$ be an $\Or^\alg$-module,
where the $\Or^\alg$-module structure is defined by the multiplication on the left component.
Define a left $\Dr$-module structure on $\Or^\alg \otimes \Mr$
by
\begin{align}
\pa_i \cdot (f\otimes \mr) = (\pa_i f) \otimes \mr + f\otimes L(-1)_i \mr
\label{eq_D_def}
\end{align}
for $i=1,\dots,r$, $f \in \Or^\alg$, $\mr \in \Mr$.


Let $N_{\Mr}$ be the $\Dr$-submodule of 
$\Or^\alg \otimes \Mr$ generated by the following elements:
\begin{align}
1 \otimes a(n)_i \mr - &\sum_{k \geq 0}\binom{n}{k} (-z_i)^k \otimes a(-k+n)_0^* \mr
+ \sum_{1\leq s\leq r, s \neq i} \sum_{k \geq 0} \binom{n}{k}(z_s-z_i)^{n-k}\otimes a(k)_s \mr,
\label{eq_ker1}
\end{align}
for all $\mr \in \Mr$, $a \in V$ and $i \in \{1,\dots,r\}$ and $n \in \Z$ (see Lemma \ref{lem_module_recursive}).
Set 
$$
D_{\Mr} = (\Or^\alg \otimes \Mr) / N_{\Mr},
$$
which is a $\Dr$-module.
The following lemma is clear:
\begin{lem}
\label{lem_D_functor}
The assignment of $\Mr$ to $D_{\Mr}$ determines the following $\C$-linear functor:
\begin{align*}
D_{\bullet}:{\Vmodu}^\op \times {\Vmodu}^r &\rightarrow \underline{\Dr \text{-mod}},\\
(M_0,M_1,\dots,M_r) \quad &\mapsto \quad D_{\Mr}.
\end{align*}
\end{lem}

Let $\Or^\an$ be the sheaf of holomorphic functions on $\Xr$.
Set
\begin{align*}
\CB_{\Mr} = \mathrm{Hom}_{\Dr}(D_{\Mr},\Or^\an),
\end{align*}
the holomorphic solution sheaf, which is called a {\bf  conformal block}.

\begin{rem}
\label{rem_assign_hol}
Let $U\subset \Xr$ be an open subset and $C\in \CB_\Mr(U)$.
Let $\Mr\rightarrow \Or^\alg \otimes \Mr$ be the embedding defined by
$\mr\mapsto 1\otimes \mr$ for $\mr\in \Mr$.
Then, by $\Mr\rightarrow \Or^\alg \otimes \Mr \rightarrow D_{\Mr}$,
$C$ can be regarded as a linear map $\Mr \rightarrow \Or^\an(U)$,
which assigns to each vector in $\Mr$ a holomorphic function on $U$.
\end{rem}

Let $M_0,M_1,\dots,M_r\in \Vmodu$.
Assume that $M_1,\dots,M_r$ are $C_1$-cofinite and $M_0^\vee$ is finitely generated.
By refining the argument in \cite{H2} (which is for the $r =3$ case), we show that $D_{\Mr}$ is a finitely generated $\Or^\alg$-module \cite[Proposition 4.6]{M1}. 
Then, by \cite[Theorem 1.4.10]{HTT}, 
$D_{\Mr}$ is a locally free $\Or^\alg$-module of finite rank as a $\Dr$-module.
Hence, we have:
\begin{prop}
\label{prop_coherence}
The holomorphic solution sheaf $\mathrm{Hom}(D_{\Mr},\Or^\an)$ is a locally constant sheaf of finite rank on $\Xr$.
\end{prop}

For a locally constant sheaf, we can define the monodromy representation of the fundamental groupoid $\Pi_1(\Xr)$ by analytic continuations.

Recall that for $A \in \Tr_r$, $T_A^\conv$ inherits a left $\Dr$-module structure. Then, 
\begin{align}
\mathrm{Hom}_{\Dr}(D_{\Mr}, T_A^\conv)
\label{eq_formal_sol}
\end{align}
is the space of all convergent formal solutions in $T_A^\conv$.
%
Since we impose the convergence in $U_A$ on the formal power series and $U_A$ is simply-connected,
any formal solutions \eqref{eq_formal_sol} 
define a well-defined section of $\CB_\Mr(U_A)$.
This gives a linear map
\begin{align*}
s_A:\mathrm{Hom}_{\Dr}(D_{\Mr}, T_A^\conv) \rightarrow \mathrm{Hom}_{\Dr}(D_{\Mr}, \Or^\an(U_A)).
\end{align*}

Conversely, we showed that 
any conformal block has an expansion of the form in $T_A^\conv$:
\begin{thm}\cite[Theorem 4.23]{M1}
\label{thm_expansion}
Let $M_0,M_1,\dots,M_r \in \Vmodf$.
For $A \in \Tr_r$, $s_A$ is an isomorphism of vector spaces. The inverse map, which is defined by the series expansion, is denoted by
$$
e_A:\CB_{\Mr}(U_A) \rightarrow \mathrm{Hom}_{\Dr}(D_{\Mr},T_A^\conv).
$$
\end{thm}

Let us consider the simplest tree $(12) \in \Tr_2$. Note that, in this case,
\begin{align*}
T_{12} = \C[z_2, z_{12}^\C, \log(z_{12})]=T_{12}^\conv
\end{align*}
is just a polynomial and does not contain any infinite series, since the set of all edges $E(12)$ is empty.
Let $M_i \in \Vmodu$ ($i=0,1,2$).
For $\Y(\bullet,z) \in I_{\log}\binom{M_0}{M_1M_2}$ and $u \in M_0^\vee$, $m_i \in M_i$,
\begin{align*}
\langle u, \exp(L(-1)z_2) \Y(m_1,z_{12})m_2 \rangle \in \C[z_2, z_{12}^\C, \log(z_{12})] = T_{12}^\conv
\end{align*}
and it is easy to show that this is an element of $\mathrm{Hom}_{D_{X_2(\C)}}(D_{M_{[0;2]}},T_{(12)}^\conv)$.
Then, we have (see \cite[Proposition 5.6]{M1}):
\begin{prop}\label{prop_int}
The above map $I_{\log}\binom{M_0}{M_1M_2} \rightarrow \CB_{M_{[0;2]}}(U_{12})$
is an isomorphism.
\end{prop}

Let $A\in \Tr_r$, $B\in \Tr_{s}$ and $p\in [r]$. Then, $A\circ_p B\in \Tr_{r+s-1}$ as in Fig. \ref{fig_tree}.
Let $M_0, M_1^A,\dots,M_r^A$ and $M_1^B,\dots,M_{s}^B$ be $V$-modules in $\Vmodf$.
Set
\begin{align*}
M_{r\circ_p s}^{A,B}= M_0^\vee \otimes  M_1^A\otimes M_2^A \otimes \cdots \otimes 
M_{p-1}^A \otimes M_1^B\otimes \cdots \otimes M_s^B \otimes M_{p+1}^A \otimes \cdots \otimes M_r^A
\end{align*}
and let
\begin{align*}
C_A \in \CB_{M_0;M_1^A,\dots,M_r^A}(U_A) \text{ and } C_B \in \CB_{M_p^A;M_1^B,\dots,M_s^B}(U_B).
\end{align*}
We recall from \cite[Section 5.3]{M1}
 the multicategorical composition (the gluing of solutions) of $C_A$ and $C_B$,
which defines a new conformal block
\begin{align*}
C_A \circ_p C_B \in \CB_{M_{r\circ_p s}^{A,B}}(U_{A \circ_pB}).
\end{align*}

By Remark \ref{rem_assign_hol} and Theorem \ref{thm_expansion}, we regard conformal blocks as
convergent formal power series valued linear maps on $M_{r\circ_p s}^{A,B}$.
For $m_{r\circ_p s}^{A,B}= u \otimes  m_1^A\otimes m_2^A \otimes \cdots \otimes 
m_{p-1}^A \otimes m_1^B\otimes \cdots \otimes m_s^B \otimes m_{p+1}^A \otimes \cdots \otimes m_r^A \in M_{r\circ_p s}^{A,B}$,
define the composition by
\begin{align}
&(C_A \circ_p C_B)(m_{[r \circ_p s]})\label{eq_comp}\\
&= \sum_{h \in \C} \Bigl(\sum_{i \in I_h} 
e_A(C_A)(u,m_1^A,\dots,m_{p-1}^A,e_i^h, m_{p+1}^A,\dots,m_r^A)
e_B(C_B)(\exp(-L(-1)^*z_{r_B})e_h^i,m_1^B,\dots,m_s^B)\Bigr) \nonumber\\
& \in \C[[\zeta_e\mid e\in E(A\circ_p B)]] [z_{A\circ_p B}, x_{A\circ_p B}^\C,\log x_{A \circ_p B},\zeta_e^\C,\log\zeta_e \mid e\in E(A\circ_p B)].\nonumber
\end{align}
Here, $\{e_i^h\}_{i \in I_h}$ is a basis of $(M_p^A)_h$ and $\{e_h^i\}_{i \in I_h}$ is the dual basis of $(M_p^A)_h^*$.
This infinite sum is well-defined as formal power series, i.e., each coefficient of formal variables is a finite sum.
It is non-trivial that the right-hand side of \eqref{eq_comp} is in $T_{A\circ_p B}^\conv$,
that is, absolutely convergent in $U_{A\circ_p B}$.
This result is obtained in \cite[Theorem 5.11]{M1}.
Hence, \eqref{eq_comp} gives a section of 
$\mathrm{Hom}_{D_{X_{r+s-1}}}(D_{M_{r\circ_p s}^{A,B}}, T_{A\circ_pB}^\conv)$,
and thus, $\CB_{M_{r\circ_p s}^{A,B}}(U_{A\circ_p B})$.

To summarize, we have the following result:
\begin{thm}
\label{thm_glue}
The following sum is absolutely convergent
\begin{align*}
\sum_{h \in \C} \Bigl|\sum_{i \in I_h} 
e_A(C_A)(u,m_1^A,\dots,m_{p-1}^A,e_i^h, m_{p+1}^A,\dots,m_r^A)
e_B(C_B)(\exp(-L(-1)^*z_{r_B})e_h^i,m_1^B,\dots,m_s^B)\Bigr|,
\end{align*}
and \eqref{eq_comp} is locally uniformly convergent in $U_{A\circ_p B}$ to an element in $\CB_{M_{r\circ_p s}^{A,B}}(U_{A\circ_p B})$.
In particular, \eqref{eq_comp} defines the linear map:
\begin{align}
\comp_{p}:\CB_{M_0;M_1^A,\dots,M_r^A}(U_A)\otimes\CB_{M_p^A;M_1^B,\dots,M_s^B}(U_B)
\rightarrow \CB_{M_{r\circ_p s}^{A,B}}(U_{A\circ_p B}).\label{eq_glue}
\end{align}
\end{thm}
This is the composition operation on the multi-hom spaces of the
pseudo-braided category discussed in the next section.

\section{Pseudo-braided category structure on conformal blocks}\label{sec_pseudo}
In this section, we review the definition of a pseudo-braided category and state Theorem A.
The collection \(\{\Tr_r\}_{r\geq 0}\) of binary trees forms the free
symmetric operad generated by one binary operation, usually called the {\it magma
operad}. (For the definition and basic properties of operads, we refer the reader to \cite{LoV,Fr}. For a more concise overview, see also \cite[Section 2]{M1}.)

More explicitly, its operad structure is described as follows:
Let $A \in \Tr_n$ and $B\in \Tr_m$ with $n> 0$ and $p \in [n]$.
The partial composition of the operad is then defined as shown in Fig. \ref{fig_tree}.
In general, $A\circ_p B$ is defined by inserting the tree $B$ into the leaf labeled with $p$ in $A$,
adding $p-1$ to labels of the leaves in $B$, and adding $m-1$ to the labels of the leaves after $p+1$ in $A$.
If $B=\emptyset$, the $p$-th leaf in $A$ is simply erased, and all subsequent leaf labels (from $p+1$ onward) are shifted down by $1$.
For example,
\begin{align*}
3((12)4) \circ_2 \emptyset = 2(13).
\end{align*}
The symmetric group $S_r$ acts on $\Tr_r$ by the permutation of labels,
which satisfies the definition of a symmetric operad.

We next look back at the definition of the parenthesized braid operad $\CPaB$ introduced in \cite{Bar,Ta1,Fr}.
Let $\cat$ be the category of small categories, i.e., objects are small categories and morphisms are functors. 
By the direct product of categories, $\cat$ has a structure of a symmetric monoidal category.
The notion of an operad can be considered in any symmetric monoidal category, and $\CPaB$ is an operad object in $\cat$.

\begin{figure}[t]
      \begin{minipage}[b]{0.45\linewidth}
    \centering
    \includegraphics[width=3.2cm]{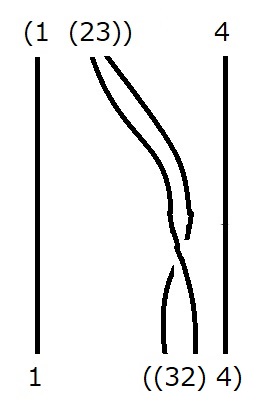}
    \caption{The morphism $\alpha\circ_2\sigma$}\label{fig_alphasigma}
  \end{minipage}
    \begin{minipage}[b]{0.45\linewidth}
    \centering
    \includegraphics[width=3.2cm]{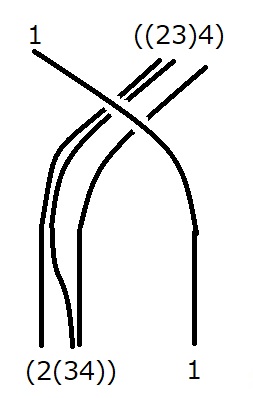}
    \caption{The morphism $\si\circ_2\al$}\label{fig_sigmaalpha}
  \end{minipage}
\end{figure}

For each $r\geq 1$, $\CPaB(r)$ is the category defined as follows:
The set of all objects in $\CPaB(r)$ is the set of binary trees $\Tr_r$,
\begin{align*}
\mathrm{Ob}(\CPaB(r))=\Tr_r.
\end{align*}

Let $p:B_r \rightarrow S_r$ be the canonical projection from the braid group to the symmetric group whose kernel is the pure braid group $PB_r$.
Denote by $g:\Tr_r\rightarrow S_r$ the map given by forgetting the parenthesization and viewing trees as permutations.
Then, for  $A,B \in \Tr_{r}$, the space of morphisms is defined by
\begin{align}
\Hom_{\CPaB(r)}(A,B)=\C p^{-1}({g_A^{-1}g_B}),
\label{eq_perm_braid}
\end{align}
where $\C p^{-1}({g_A^{-1}g_B})$ is a $\C$-linear space with a basis $p^{-1}({g_A^{-1}g_B})$.
The composition law is induced from the one on $B_r$.
The symmetric group $S_r$ acts on $\CPaB(r)$ via renumbering the objects $\Tr_r$ and acts identically on morphisms.
The composition
\begin{align*}
\circ_p: \CPaB(n)\times \CPaB(m)\rightarrow \CPaB(n+m-1)
\end{align*}
is given by replacing the $p$-th strand of the first braid, by the second braid made very thin (see Fig. \ref{fig_thin} in the introduction). This composition is consistent with the magma operad when restricted to objects.

For $r=0$, $\CPaB(0)$ is a category whose object is the only empty parenthesized word $\emptyset$ and whose 
morphism consists only of the identity map $\Hom(\emptyset,\emptyset)=\C\{\id\}$.
The composition
\begin{align*}
\circ_p: \CPaB(n)\times \CPaB(0)\rightarrow \CPaB(n-1)
\end{align*}
is given by just erasing the $p$-th strand.

An important point here is that the operad $\CPaB$ is generated by $\sigma \in \CPaB(2)$ and $\alpha \in \CPaB(3)$ (see Fig. \ref{fig_sigma}, Fig.~\ref{fig_alpha}, Fig.~\ref{fig_alphasigma}, and Fig.~\ref{fig_sigmaalpha}). These elements correspond to the braiding and the associator in a braided tensor category, respectively. 
The defining relations among these generators are precisely the operadic
forms of the pentagon and hexagon identities. This is the reason why an
action of \(\CPaB\) on representable multi-hom spaces gives rise to a braided
tensor category (see \cite{Fr} or \cite[Section 2]{M1}).


A {\it pseudo-tensor category}, also known as a {\it multicategory} (see for example \cite{Lam, BD}), is a generalization of
a tensor category in which the tensor product is not necessarily
representable. In an ordinary category, morphisms have one input and one
output: $\Hom(M;L)$.
In a pseudo-tensor category, this is replaced by {\it multi-hom spaces}
\[
  \Hom(M_1,\ldots,M_r;L),
\]
where \(M_1,\ldots,M_r\) are input objects and \(L\) is an output object,
together with composition maps satisfying the usual associativity axioms.
For example, a strict tensor category \((\cC,\boxtimes,\mathbf 1)\) gives a
pseudo-tensor category by setting
\[
  \Hom(M_1,\ldots,M_r;L)
  =
  \Hom_\cC(M_1\boxtimes\cdots\boxtimes M_r,L).
\]
Thus a pseudo-tensor category may be viewed as a structure in which the
multi-input operations are specified directly, without assuming that they are
represented by tensor products.

Soibelman introduced in \cite{So} a braided analogue of this notion, called a
pseudo-braided category. In \cite{M1}, we use a unital version formulated in
terms of the parenthesized braid operad. We recall here only the part of the
definition needed in this note; for the precise definition, see
\cite[Section~2.3]{M1}.

A unital pseudo-braided category \(\mathcal C\) has a class of objects \(\operatorname{Ob}(\mathcal C)\), a
distinguished unit object \(\mathbf 1\in \operatorname{Ob}(\mathcal C)\), and,
for each \(r\geq 1\), \(A\in T_r\), and objects
\(M_0,M_1,\ldots,M_r\), a vector space
\[
\operatorname{Hom}_A(M_1,\ldots,M_r;M_0).
\]

%
For \(A\in\Tr_n\), \(B\in\Tr_m\), and \(1\leq p\leq n\), there are composition
maps
\[
\begin{aligned}
&\circ_p:\Hom_A(N_1,\ldots,N_{p-1},L,N_{p+1},\ldots,N_n;R)
\otimes
\Hom_B(M_1,\ldots,M_m;L)
\\
&\hspace{4em}\longrightarrow
\Hom_{A\circ_p B}
(N_1,\ldots,N_{p-1},M_1,\ldots,M_m,N_{p+1},\ldots,N_n;R).
\end{aligned}
\]
These maps, together with the unit insertion maps, satisfy the usual associativity and unit axioms.
In addition, the parenthesized braid operad acts on the multi-hom spaces. Thus, for any morphism \(g:A\to B\) in \(\PaB(r)\), there is an
isomorphism
\[
  \rho(g):
  \Hom_A(M_1,\ldots,M_r;M_0)
  \longrightarrow
  \Hom_{B}(M_1,\ldots,M_r;M_0),
\]
functorial in \(g\). 
For readability, we use the symbol \(\bullet\) in diagrams below to suppress
the list of input and output objects.
The action of \(\PaB\) is compatible with the operadic composition. Namely,
for morphisms \(g_A:A\to A'\) in \(\PaB(n)\) and \(g_B:B\to B'\) in
\(\PaB(m)\), the following diagram commutes:
\[
\begin{CD}
\Hom_A(\bullet;\bullet)\otimes \Hom_B(\bullet;\bullet)
@>{\circ_p}>>
\Hom_{A\circ_p B}(\bullet;\bullet)
\\
@V{\rho(g_A)\otimes\rho(g_B)}VV
@VV{\rho(g_A\circ_p g_B)}V
\\
\Hom_{A'}(\bullet;\bullet)\otimes \Hom_{B'}(\bullet;\bullet)
@>{\circ_p}>>
\Hom_{A'\circ_p B'}(\bullet;\bullet).
\end{CD}
\]

Every unital non-strict braided tensor category \(\mathcal C\)
gives rise to a unital pseudo-braided category by
\begin{align*}
\mathrm{Hom}_A(M_1,\dots,M_r;L) = \mathrm{Hom}_C(\boxtimes_A M_{[r]},L)
\end{align*}
for any tree $A\in \Tr_r$ with the canonical \(\PaB\)-action induced by the associator and braiding (see \cite{Mac}),
where $\boxtimes_A$ is the composite of the tensor product $\boxtimes$ of shape $A$.
In this sense, a unital pseudo-braided category is a generalization of a
braided tensor category in which the multi-hom functors need not be
represented by tensor products. Conversely, \cite[Theorem~2.19]{M1} gives
necessary and sufficient conditions under which a unital pseudo-braided
category is represented by an actual braided tensor category.

We now state Theorem A in this setting.
Let $V$ be a vertex operator algebra. Assume that $V^\vee$ is a finitely generated $V$-module.
Define a multi-morphism space of $\Vmodf$ as follows:
\begin{align*}
\mathrm{Hom}_A(M_1,\dots,M_r;M_0) = \CB_{\Mr}(U_A).
\end{align*}
Then, the composition maps constructed in Theorem \ref{thm_glue}, together with analytic
continuation along morphisms in $\CPaB$, define a unital pseudo-braided
category structure on \(\Vmodf\) with unit object \(V\).

Furthermore, when $V$ is rational and $C_2$-cofinite, the bifunctor 
\[
\boxtimes: \Vmodf \times \Vmodf \rightarrow \Vmodf
\]
is defined by \eqref{def_otimes}, and for any $r \geq 0$, $M_i \in \Vmodf$, and binary tree $A$, there are natural isomorphisms
\begin{align}
\mathrm{Hom}_{\Vmodf}(\boxtimes_A (M_1,\dots,M_r),M_0) \cong \CB_{M_{[0;r]}}(U_A)\label{eq_natural_iso}
\end{align}
as shown in \cite[Theorem 7.22 and Theorem 7.23]{M1}.
Thus, analytic continuations on conformal blocks from \(U_A\) to \(U_{A'}\)
induce, by the Yoneda isomorphism, morphisms of representing objects in the opposite
direction. Based on this, we define the braiding $B_{M_1,M_2}: M_1 \boxtimes M_2 \rightarrow M_2 \boxtimes M_1$ by the commutative diagram
\begin{align}
\begin{split}
\begin{array}{ccc}
\mathrm{Hom}_{\Vmodf}(M_2 \boxtimes M_1,M_0)      &\overset{\cong}{\longrightarrow}& \CB_{M_0;M_1,M_2}(U_{21})\\
    {}_{B_{M_1,M_2}^*}\downarrow 
      && 
    \downarrow_{\rho(\si)^{-1}}
    \\
\mathrm{Hom}_{\Vmodf}(M_1 \boxtimes M_2,M_0)      &\overset{\cong}{\longrightarrow}& \CB_{M_0;M_1,M_2}(U_{12}),
\end{array}
\end{split}
\label{eq_def_braid}
\end{align}
and the associator $\al_{M_1,M_2,M_3}:(M_1\boxtimes M_2)\boxtimes M_3 \rightarrow M_1 \boxtimes (M_2\boxtimes M_3)$ by
\begin{align}
\begin{split}
\begin{array}{ccc}
\mathrm{Hom}_{\Vmodf}(M_1\boxtimes (M_2 \boxtimes M_3),M_0)      &\overset{\cong}{\longrightarrow}& \CB_{M_0;M_1,M_2,M_3}(U_{1(23)})\\
    {}_{\al_{M_1,M_2,M_3}^*}\downarrow 
      && 
    \downarrow_{\rho(\al)^{-1}}
    \\
\mathrm{Hom}_{\Vmodf}((M_1\boxtimes M_2)\boxtimes M_3),M_0)      &\overset{\cong}{\longrightarrow}& \CB_{M_0;M_1,M_2,M_3}(U_{(12)3}).
\end{array}
\end{split}
\label{eq_def_associator}
\end{align}
Here \(B_{M_1,M_2}^{*}\) and
\(\alpha_{M_1,M_2,M_3}^{*}\) denote the pullback maps induced by
\(B_{M_1,M_2}\) and \(\alpha_{M_1,M_2,M_3}\) on the corresponding Hom-spaces.
These definitions endow $\Vmodf$ with the structure of a balanced braided tensor category, with tensor product $\boxtimes$, braiding $B$, and twist $\theta_M= \exp(2\pi i L(0))$.
In the next section, we illustrate this construction for the Virasoro vertex
operator algebra $L(\ft,0)$, where the analytic continuations can be computed
explicitly and the resulting category is identified with a Tambara--Yamagami category.

\begin{rem}\label{rem_path_explicit}
When realizing morphisms in $\CPaB$ by paths in configuration spaces and
acting on conformal blocks by analytic continuation, 
there is a choice of whether the generator \(\sigma\) (as in Fig. \ref{fig_sigma}) should be represented
by a clockwise or counterclockwise half-turn in $X_2(\C)$.
This corresponds to the general fact that for any braided tensor category $\cC$, one can define a new braided tensor category $\cC^{\mathrm{rev}}$ by reversing the braiding:
\[
B^{\mathrm{rev}}_{M,N} = B^{-1}_{N,M}: M \boxtimes N \to N \boxtimes M,
\]
where $B_{M,N}$ denotes the original braiding of $\cC$.
In this note we choose the clockwise representative. This is the
choice compatible with the standard balancing isomorphism
\[
  \theta_M=\exp(2\pi iL(0)).
\]
%
More precisely, we choose a point $(a_1,a_2) \in X_2(\mathbb{C})$ with $a_1, a_2 \in \mathbb{R}$ such that $a_1 > a_2 > 0$. Then, $(a_1,a_2) \in U_{12}$, and we define a path
\[
\sigma: [0,1] \rightarrow X_2(\mathbb{C})
\]
as follows:
\begin{align}
\si(t) = \left(\frac{a_1+a_2}{2}+\frac{a_1-a_2}{2}e^{-\pi i t}, \frac{a_1+a_2}{2}-\frac{a_1-a_2}{2}e^{-\pi i t}\right) \in X_2(\C). \label{eq_path_si}
\end{align}
Then, $\si(0)=(a_1,a_2) \in U_{12}$ and 
$\si(1) = (a_2,a_1) \in U_{21}$. In particular, in $(12)$-coordinate 
\begin{align*}
x_{12}(\si(t))= (a_1-a_2)e^{-\pi i t}.
\end{align*}
\end{rem}

\section{An example: Virasoro VOA of central charge $c=\tfrac{1}{2}$}
Let $L(\ft,0)$ be the simple Virasoro vertex operator algebra of central charge $\ft$ \cite{FZ}. 
It is well known that $L(\ft,0)$ is rational and $C_2$-cofinite \cite{Wa,DLM}. 
Hence, by Theorem \ref{thm_B}, the representation category $\Lmod$ inherits a braided tensor category structure 
(see also \cite{HL4,HL1,HL2,HL3,H1,H2,H3}). 
Note that since $L(\ft,0)$ is $C_2$-cofinite, any finitely generated module is also $C_2$-cofinite and therefore $C_1$-cofinite.

The isomorphism classes of irreducible modules for $L(\ft,0)$ are given by $L(\ft,0)$, $L(\ft,\ft)$, and $L(\ft,\fs)$ 
(see, for example, \cite{IK}), which are distinguished by the $L(0)$-weights of their lowest weight vectors, namely,
\begin{align*}
L\left(\ft,h\right) = \bigoplus_{n \geq 0} L\left(\ft,h\right)_{h+n}, \quad\quad \left(h = 0, \ft, \fs\right),
\end{align*}
where $L(\ft,h)_\Delta$ denotes the $L(0)$-eigenspace of weight $\Delta \in \mathbb{R}$.

To determine the braided tensor category structure, it suffices to compute the conformal blocks. 
This computation was essentially carried out in \cite{BPZ} (although not all conformal blocks were explicitly written down therein).

In this section, we explicitly compute all four-point conformal blocks for $L(\ft,0)$, and thereby show that the braided tensor category $\Lmod$ is equivalent to a Tambara-Yamagami tensor category over $\mathbb{Z}_2$ as a balanced braided tensor category (for the definition of a balanced braided tensor category, see \cite{EGNO}).

A Tambara-Yamagami tensor category $\TY(A,\chi,\tau)$ is a fusion category defined from the data of a finite abelian group $A$, 
a symmetric nondegenerate bicharacter $\chi: A \times A \rightarrow \mathbb{C}^\times$, and an element $\tau \in \mathbb{C}$ satisfying $|A| \tau^2 = 1$ \cite{TY}. 

In our case, the relevant data are given by
\begin{align}
A = \mathbb{Z}_2, \quad \chi(1,1) = -1, \quad \tau = 2^{-1/2}. \label{eq_TY_data}
\end{align}
The Tambara-Yamagami tensor category $\mathrm{TY}(A, \chi, \tau)$ is defined as follows:
%
%
\begin{itemize}
\item
Objects are finite direct sums of elements of $S=A \sqcup m$.
\item
Hom-sets between elements of $S$ are given by
\begin{align*}
\mathrm{Hom}(s,s')= \begin{cases}
\C& \text{ if $s=s'$},\\
0& \text{ if $s\neq s'$}.
\end{cases}
\end{align*}
\item
Compositions are obvious ones and $\id_s = 1 \in \C$.
\item
Tensor products of elements of $S$ are given by
\begin{align*}
a\otimes b = ab,\quad a \otimes m = m,\quad 
m \otimes a = m,\quad m \otimes m =\bigoplus_{a\in A}a
\end{align*}
and the unit object is $1 \in A$.
\item
The associator $\al$ for elements of $S$ are given by
\begin{align*}
\al_{a,b,c}=\id_{abc}&:abc=(a\otimes b)\otimes c \rightarrow a\otimes (b\otimes c)\\
\al_{a,b,m}=\al_{m,a,b}=\id_{m}&:m \rightarrow m\\
\al_{a,m,b} =\chi(a,b)\id_m &: m \rightarrow m\\
\al_{a,m,m} =\al_{m,m,a}=\bigoplus_{b\in A} \id_b &: \bigoplus_{b\in A}b \rightarrow \bigoplus_{b\in A}b\\
\al_{m,a,m} =\bigoplus_{b\in A} \chi(a,b) \id_b &: \bigoplus_{b\in A}b \rightarrow \bigoplus_{b\in A}b\\
\al_{m,m,m} =(\tau \chi(a,b)^{-1}\id_m)_{a,b} &: \bigoplus_{a\in A}m \rightarrow \bigoplus_{b\in A}m
\end{align*}
with $a,b,c\in A$. The left and right unitors are identity maps.
\end{itemize}
Note that in our case
$\al_{m,m,m}: \bigoplus_{a\in \Z_2} m \rightarrow \bigoplus_{b \in \Z_2} m$ is equal to
\begin{align}
\frac{1}{\sqrt{2}}
\begin{pmatrix}
1 &1\\
1 & -1\\
\end{pmatrix}.
\label{eq_al_mmm}
\end{align}

%
%

We now consider the case given in \eqref{eq_TY_data}. Set
\begin{align*}
h_0 = 0, \quad h_1 = \ft,\quad h_m=\fs.
\end{align*}
Then, for any $a,b \in \mathbb{Z}_2$, we have
\[
\chi(a,b) = \exp\left(\pi i (h_{a+b} - h_a - h_b)\right).
\]

Under this choice of data, the category $\TY(\mathbb{Z}_2, \chi, 2^{-1/2})$ can be endowed with a balanced braided tensor category structure by defining the twist and the braiding as follows:
\begin{align*}
\theta_s = \exp(2\pi i h_s) \cdot \mathrm{id}_s : s \rightarrow s 
\end{align*}
for $s \in S = \Z_2 \sqcup m$ and
\begin{align}
\begin{split}
\sigma_{a,b} = \exp\left(\pi i (h_{a+b} - h_a - h_b)\right) \cdot \mathrm{id}_{ab} 
&: ab = a \otimes b \rightarrow b \otimes a, \\
\sigma_{a,m} = \sigma_{m,a} = \exp\left(\pi i(h_m- h_a-h_m)\right) \cdot \mathrm{id}_m 
&: m = a \otimes m \rightarrow m \otimes a, \\
\sigma_{m,m} = \left( \delta_{a,b} \cdot \exp\left( \pi i \left(h_a - h_m-h_m \right) \right) \right)_{a,b}
&: \bigoplus_{a \in \mathbb{Z}_2} a = m \otimes m \rightarrow \bigoplus_{b \in \mathbb{Z}_2} b = m \otimes m.
\end{split}
\label{eq_TY_braid}
\end{align}

We now compare the category $\TY(\mathbb{Z}_2, \chi, 2^{-1/2})$ with $\Lmod$. 
As a first step, let us recall some results about the intertwining operators among modules of $L(\ft,0)$
(This is equivalent to studying conformal blocks on $X_2(\C)$, by Proposition \ref{prop_int}).
Set 
\begin{align*}
\Is=\left\{0,\ft,\fs\right\}
\end{align*}
and $I\binom{h_0}{h_1h_2}=I\binom{L(\ft,h_0)}{L(\ft,h_1)L(\ft,h_2)}$ for $h_i \in \Is$, the space of intertwining operators among $L(\ft,0)$-modules.
For $h\in\Is$, let $\ket{h} \in L\left(\ft,h\right)_{h}$ be a non-zero vector and $\bra{h} \in L\left(\ft,h\right)_h^* \subset L\left(\ft,h\right)^\vee$ be the unique linear map satisfying $\bra{h}\ket{h}=1$.
Then, the following proposition is obtained in \cite{DMZ}:
\begin{prop}
\label{fusion_ising}
For $h_0,h_1,h_2 \in \Is$,
\begin{align*}
\dim I\binom{h_0}{h_1h_2}&=
\begin{cases}
1, & h_0 \in h_1 \star h_2,\\
0, & \text{otherwise}
\end{cases}
\end{align*}
where $\star$ is a map from $\Is \times \Is$ to the power set $P(\Is)$ which is described as follows:
\begin{table}[h]
\label{table_fusion}
\caption{Fusion rule}
  \begin{tabular}{|c|ccc|} \hline
$\star$ & $0$ & $\ft$ & $\fs$ \\ \hline
$0$ & $0$ & $\ft$ & $\fs$ \\
$\ft$ & $\ft$ & $0$ & $\fs$ \\
$\fs$ & $\fs$ & $\fs$ & $\{0, \ft\}$ \\ \hline
\end{tabular}
\end{table}

Furthermore,
if $h_0 \in h_1 \star h_2$ and $\Y(-,z) \in  I\binom{h_0}{h_1h_2} \setminus \{0\}$,
then $\bra{h_0}Y(\ket{h_1},z)\ket{h_2}\neq 0$.
\end{prop}


By \eqref{def_otimes}, the tensor product in $\Lmod$ is given by
\[
  M_1\boxtimes M_2
  =
  \bigoplus_{h\in IS}
  L\left(\frac12,h\right)\otimes_\C
  I\binom{L(\ft,h)}{M_1\,M_2}^{*}
\]
for \(M_1,M_2\in \Lmod\).

We next explain how the braiding is obtained from analytic continuation. Let
\[
  R_{h;M_1,M_2}:
  I\binom{L(\ft,h)}{M_1\,M_2}
  \longrightarrow
  I\binom{L(\ft,h)}{M_2\,M_1}
\]
be the linear isomorphism induced by analytic continuation along the clockwise
representative of
$\sigma:(12)\longrightarrow(21)$
fixed in Remark \ref{rem_path_explicit}. Equivalently, \(R_{h;M_1,M_2}\) is the composition
\[
\begin{aligned}
I\binom{L(\ft,h)}{M_1\,M_2}
\xrightarrow{\;s_{12}\;}
\CB_{L(\ft,h);M_1,M_2}(U_{12})
\xrightarrow{\;A(\sigma)\;}
\CB_{L(\ft,h);M_1,M_2}(U_{21})
\xrightarrow{\;e_{21}\;}
I\binom{L(\ft,h)}{M_2\,M_1}.
\end{aligned}
\]
Here \(s_{12}\) and \(e_{21}\) are the standard identifications between
intertwining operators and conformal blocks in Section \ref{sec_chiral_CB}.
By \eqref{eq_def_braid}, the categorical braiding
$B_{M_1,M_2}:M_1\boxtimes M_2\longrightarrow M_2\boxtimes M_1$
is given by
\[
  B_{M_1,M_2}
  =
  \bigoplus_{h\in IS}
  \id_{L(\ft,h)}
  \otimes
  \left((R_{h;M_1,M_2})^{-1}\right)^*.
\]
Thus, after bases of the one-dimensional intertwining-operator spaces are
chosen, the scalar appearing in the braiding is the inverse of the scalar
appearing in the analytic continuation \(R_{h;M_1,M_2}\).

Therefore, we begin by fixing a basis for each $I\binom{h_0}{h_1 h_2}$. Of course, the braided tensor category structure is independent of such choices of basis.
Let $h_0,h_1,h_2 \in \Is$ with $h_0 \in h_1\star h_2$
and $\Y(-,z) \in I\binom{h_0}{h_1h_2}$.
By (P1) and (P3), we have
\begin{align*}
z\frac{d}{dz}  \langle u_0, \Y(a_1,z)a_2  \rangle &=
\langle u_0, [L(0),\Y(a_1,z)]a_2 \rangle - \langle u_0, \Y(L(0)a_1,z)a_2  \rangle,\\
 &= (\Delta_0-\Delta_1-\Delta_2)  \langle u_0, \Y(a_1,z)a_2  \rangle
\end{align*}
for $u_0 \in L(\ft,h_0)_{\Delta_0}^*$ and $a_i \in L(\ft,h_i)_{\Delta_i}$ ($i=1,2$). Hence, we have
\begin{align}
\langle u_0, \Y(a_1,z)a_2 \rangle \in \C z^{h_0-h_1-h_2}.\label{eq_example_braid}
\end{align}
Let ${I'}_{h_1h_2}^{h_0}(-,z) \in I\binom{h_0}{h_1h_2}$ be an intertwining operator
satisfying
$$
\bra{h_0}{I'}_{h_1h_2}^{h_0}(\ket{h_1},z)\ket{h_2}=z^{h_0-h_1-h_2}.
$$
Such an intertwining operator uniquely exists by Proposition \ref{fusion_ising}.
Then,
$$
{I'}_{h_1h_2}^{h_0}(\ket{h_1},z)\ket{h_2}=z^{h_0-h_1-h_2}\ket{h_0} +\text{higher terms} \in z^{h_0-h_1-h_2} L\left(\ft,h_0\right)[[z]].
$$
By fixing these bases, the analytic continuation along $\si$ in $\Pi_1(X_2(\C))$
and $\al$ in $\Pi_1(X_3(\C))$ can be expressed not merely as linear maps but as complex numbers.
However, such numerical values depend on the specific choice of bases.
In order to extract data that matches that of $\TY(\Z_2,\chi,2^{-1/2})$, we must fix the bases as follows:
\begin{align}
{I}_{h_1h_2}^{h_0}(-,z)= 
\begin{cases}
\sqrt{2}^{-1}{I'}_{h_1h_2}^{h_0}(-,z), & (h_0,h_1,h_2) \text{ is a permutation of } (\frac{1}{2},\frac{1}{16},\frac{1}{16}), \\
{I'}_{h_1h_2}^{h_0}(-,z), & \text{otherwise}.
\end{cases}
\label{eq_normal}
\end{align}

%
%

Now, let us compute, for example, the braiding
\[
  L\left(\frac12,\frac1{16}\right)
  \boxtimes
  L\left(\frac12,\frac12\right)
  \rightarrow
  L\left(\frac12,\frac12\right)
  \boxtimes
  L\left(\frac12,\frac1{16}\right).
\]
Since
\[
  L\left(\frac12,\frac1{16}\right)
  \boxtimes
  L\left(\frac12,\frac12\right)
  =
  L\left(\frac12,\frac1{16}\right)
  \otimes_\C
  I\binom{1/16}{1/16\;\;1/2}^{*},
\]
this braiding is determined by the analytic continuation of the
one-dimensional space
\[
  I\binom{1/16}{1/16\;\;1/2} =\C I^{1/16}_{1/16,1/2}(-,z).
\]
Let $\left(I^{1/16}_{1/16,1/2}\right)^{\vee}$ be the dual basis of
\(I\binom{1/16}{1/16\;\;1/2}^{*}\). Since the leading term is proportional to
\(\sqrt{2}^{-1}z^{-1/2}\), analytic continuation along \(\sigma\) in Remark \ref{rem_path_explicit} gives
\[
  R_{1/16;\,1/16,1/2}
  \left(I^{1/16}_{1/16,1/2}\right)
  =
  e^{\pi i/2}
  I^{1/16}_{1/2,1/16}.
\]

Since the categorical braiding is the transpose of the inverse continuation,
we obtain
\[
  B_{1/16,1/2}
  \left(
    v\otimes
    \left(I^{1/16}_{1/16,1/2}\right)^{\vee}
  \right)
  =
  e^{-\pi i/2}
  v\otimes
  \left(I^{1/16}_{1/2,1/16}\right)^{\vee}.
\]
This agrees with the braiding of the Tambara--Yamagami category
\(\mathrm{TY}(\mathbb Z_2,\chi,2^{-1/2})\), since
$e^{-\pi i/2}
  =
  \exp\left(\pi i
  \left(
    \fs - \fs -\ft
  \right)\right).
$

We next examine the associator. The relevant conformal block spaces are
\[
  \mathcal H^{(12)3}_{h_0;h_1,h_2,h_3}
  =
  \bigoplus_{h\in h_1\star h_2}
  I\binom{h_0}{h\;\;h_3}
  \otimes_\C
  I\binom{h}{h_1\;\;h_2}
\]
and
\[
  \mathcal H^{1(23)}_{h_0;h_1,h_2,h_3}
  =
  \bigoplus_{h\in h_2\star h_3}
  I\binom{h_0}{h_1\;\;h}
  \otimes_\C
  I\binom{h}{h_2\;\;h_3}.
\]
Let
\[
  F_{h_0;h_1,h_2,h_3}:
  \mathcal H^{(12)3}_{h_0;h_1,h_2,h_3}
  \longrightarrow
  \mathcal H^{1(23)}_{h_0;h_1,h_2,h_3}
\]
be the linear isomorphism obtained by analytic continuation along $\alpha:(12)3\longrightarrow 1(23)$.
Under the decompositions
\[
  (M_1\boxtimes M_2)\boxtimes M_3
  \cong
  \bigoplus_{h_0\in IS}
  L\left(\frac12,h_0\right)
  \otimes_\C
  \left(
    \mathcal H^{(12)3}_{h_0;h_1,h_2,h_3}
  \right)^*
\]
and
\[
  M_1\boxtimes(M_2\boxtimes M_3)
  \cong
  \bigoplus_{h_0\in IS}
  L\left(\frac12,h_0\right)
  \otimes_\C
  \left(
    \mathcal H^{1(23)}_{h_0;h_1,h_2,h_3}
  \right)^*,
\]
the categorical associator is
\[
  \alpha_{M_1,M_2,M_3}
  =
  \bigoplus_{h_0\in IS}
  \id_{L(1/2,h_0)}
  \otimes
  \left(F_{h_0;h_1,h_2,h_3}^{-1}\right)^*.
\]
Thus, in the computations below, the connection matrices obtained by
re-expanding conformal blocks must be interpreted through this inverse
transpose.

Hereafter, following \cite{BPZ}, we explicitly compute the connection matrix.
For $h_0,h_1,h_2,h_3 \in \Is$ and $h \in h_1 \star h_2$, define
\begin{align*}
{C'}_{h_0;h_1h_2h_3}^h(x,y)
&= \bra{h_0} {I'}_{h_1 h}^{h_0}(\ket{h_1},x) {I'}_{h_2 h_3}^{h}(\ket{h_2},y)  \ket{h_3},\\
C_{h_0;h_1h_2h_3}^h(x,y)
&= \bra{h_0} I_{h_1 h}^{h_0}(\ket{h_1},x) I_{h_2 h_3}^{h}(\ket{h_2},y)  \ket{h_3}.
\end{align*}
Since
\begin{align*}
L\left(\ft,h\right) = \mathrm{Span}_\C \{L(-i_1-1) \dots L(-i_N-1) \ket{h} \mid i_1 \geq \dots \geq i_N \geq 0 \},
\end{align*}
the recursion relation \eqref{eq_ker1} reduces all matrix coefficients of a conformal
block to its values on lowest-weight vectors. Hence the conformal block is
determined by the functions
$C^h_{h_0;h_1h_2h_3}(x,y)$.

Considering the $L(0)$-action, we find that $C_{h_0;h_1h_2h_3}^h(x,y)$ satisfies the differential equation (see \cite[Section 4.6]{M1})
\begin{align*}
(x\pa_x + y \pa_y +h_1+h_2+h_3-h_0)C_{h_0;h_1,h_2,h_3}^{h}(x,y)=0,
\end{align*}
from which we immediately obtain
\begin{align*}
{C}_{h_0,h_1,h_2,h_3}^{h}(x,y)
= x^{h_0-h_1-h_2-h_3} {C}_{h_0,h_1,h_2,h_3}^{h}\left(\frac{y}{x}\right),
\end{align*}
where
\begin{align}
{C}_{h_0,h_1,h_2,h_3}^{h}(z)
= \lim_{x \to 1} {C}_{h_0,h_1,h_2,h_3}^{h}(x,z)
= \bra{h_0} {I}_{h_1, h}^{h_0}(\ket{h_1}, 1)
{I}_{h_2,h_3}^h (\ket{h_2}, z) \ket{h_3},
\label{eq_cb_limit}
\end{align}
which defines a formal power series in $z^{h - h_2 - h_3} \C[[z]]$.
Similarly to \eqref{eq_example_braid}, we find
\begin{align}
{C'}_{h_0,h_1,h_2,h_3}^{h}(x,y)
\in y^{h-h_2-h_3} x^{h_0-h_1-h} \left(1+\frac{y}{x}\C\left[\left[\frac{y}{x}\right]\right]\right).
\label{eq_CB_leading}
\end{align}

We have thus reduced the computation of conformal blocks to the calculation of ${C}_{h_0,h_1,h_2,h_3}^{h}(z)$.
Belavin-Polyakov-Zamolodchikov \cite{BPZ} further discovered that ${C}_{h_0,h_1,h_2,h_3}^{h}(z)$ can be completely determined by the recursion relations. We now explain their argument based on the definition of conformal blocks.
It is convenient to set
\begin{align*}
h_a = \frac{3}{4}a - \frac{1}{2},
\end{align*}
for $a \in \left\{ \frac{3}{4}, \frac{4}{3} \right\}$.
Then $h_{4/3} = \frac{1}{2}$ and $h_{3/4} = \frac{1}{16}$, and
\begin{align}
L(-1)\ket{0} = 0, \quad 
(L(-1)^2 - a L(-2)) \ket{h_a} = 0
\label{eq_singular_vect}
\end{align}
hold for $a \in \left\{ \frac{3}{4}, \frac{4}{3} \right\}$.

\begin{table}[h]
\caption{Differential equations}
\label{table_differential}
\begin{tabular}{|l|l|} \hline
$(h_0,h_1,h_2,h_3)$ & $D_{h_0,h_1,h_2,h_3}$ \\ \hline \hline
$(\frac{1}{16},\frac{1}{16},\frac{1}{16},\frac{1}{16})$
& $z(1 - z)\pa_z^2 + (3/4 - 3/2 z) \pa_z - 3/(64 z (1 - z))$ \\ \hline
$(\frac{1}{2},\frac{1}{2},\frac{1}{2},\frac{1}{2})$
& $z (1 - z)\pa_z^2+ (4/3 - 8/3 z)\pa_z- 2/(3 z (1 - z))$ \\ \hline
$(\frac{1}{2},\frac{1}{2},\frac{1}{16},\frac{1}{16})$
& $z (1 - z) \pa_z^2 + (3/4 - 3/2 z) \pa_z - (3 + 21z^2)/(64z (1 -z))$ \\
$(\frac{1}{2},\frac{1}{16},\frac{1}{2},\frac{1}{16})$
& $z (1 - z) \pa_z^2 + (3/4 - 3/2 z) \pa_z - 3/(8z (1 - z))$ \\
$(\frac{1}{16},\frac{1}{2},\frac{1}{2},\frac{1}{16})$
& $z (1 - z) \pa_z^2 + (3/4 - 13/4 z) \pa_z - (3 -7 z^2)/(8z (1 - z))$ \\
$(\frac{1}{2},\frac{1}{16},\frac{1}{16},\frac{1}{2})$
& $z (1 - z) \pa_z^2 + (4/3 - 11/12 z) \pa_z - (16 +21 z^2)/(192 z (1 - z))$ \\
$(\frac{1}{16},\frac{1}{2},\frac{1}{16},\frac{1}{2})$
& $z (1 - z) \pa_z^2+ (4/3 - 8/3 z) \pa_z - 1/(12z (1 - z))$ \\
$(\frac{1}{16},\frac{1}{16},\frac{1}{2},\frac{1}{2})$
& $z (1 - z)\pa_z^2 + (4/3 - 8/3 z) \pa_z - (8 -7 z^2)/(12z (1 - z))$ \\
\hline
\end{tabular}
\end{table}

Then, by the recursion relation and \eqref{eq_singular_vect}, we have:

\begin{lem}[{\cite{BPZ}}]\label{block_differential}
Let $h_0,h_1,h_2 \in \Is$ and $a \in \left\{ \frac{4}{3}, \frac{3}{4} \right\}$ and $h \in h_2 \star h_a$.
Set
\begin{align*}
D_{h_0,h_1,h_2,h_a} = 
(1-z)z \pa_z^2 + ((2\tilde{h} - 2 + a)z + a)\pa_z + \frac{1}{(1-z)z}
((\tilde{h}(a+\tilde{h}-1) - a h_1)z^2 - a h_2),
\end{align*}
where $\tilde{h} = h_0 - h_1 - h_2 - h_a$.
Then, ${C}_{h_0,h_1,h_2,h_a}^{h}(z)$ satisfies
\begin{align*}
D_{h_0,h_1,h_2,h_a}{C}_{h_0,h_1,h_2,h_a}^{h}(z)=0,
\end{align*}
where the differential operator
$D_{h_0,h_1,h_2,h_a}$ is listed in Table \ref{table_differential}.
\end{lem}

\begin{proof}
By the recursion relation,
\begin{align*}
0 &= \bra{h_0} {I'}_{h_1, h}^{h_0}(\ket{h_1}, x){I'}_{h_2,h_a}^h (\ket{h_2}, y) \bigl(L(-1)^2-a L(-2)\bigr) \ket{h_a} \\
&= \left(\pa_x+\pa_y\right)^2 + a\left( x^{-1}\pa_x + y^{-1}\pa_y - \frac{h_1}{x^2} - \frac{h_2}{y^2} \right) 
{C}_{h_0,h_1,h_2,h_a}^{h}(x,y),
\end{align*}
from which the assertion follows by taking the limit $x \to 1$ using \eqref{eq_cb_limit}.
\end{proof}

It is easy to verify that these are second-order regular differential equations with possible singularities at $\{0,1,\infty\}$.  
Therefore, they reduce to Gauss's hypergeometric differential equations, and all their solutions can be expressed in terms of hypergeometric functions (see for example \cite{WW}). Then, we have:
\begin{thm}
\label{list_block2}
The Virasoro conformal blocks $C_{h_0,h_1,h_2,h_3}^{h}(x,y)$ are given in Table \ref{table_vacuum} and Table \ref{table_block2} (see below for the choice of branches).
\end{thm}
\begin{proof}
The solution space of the Gauss hypergeometric differential equation on a simply connected domain is two-dimensional. 
In this case, since the exponents of the differential equation are non-degenerate, the solution is uniquely determined by the leading order 
$C_{h_0,h_1,h_2,h_3}^{h}(z) \sim z^{h - h_2 - h_3}$ near $z = 0$ (see \eqref{eq_CB_leading}).
\end{proof}


\begin{table}[h]
\caption{Virasoro conformal blocks with vacuum}
\label{table_vacuum}
  \begin{tabular}{|c|c|} \hline
 & ${C'}_{h_0,h_1,h_2,h_3}^{h}(x,y)$ \\ \hline
$h_0=0$ & $x^{h_2-h_1-h_3}y^{h_1-h_2-h_3}(x-y)^{h_3-h_1-h_2}$\\
$h_1=0$ & $y^{h_0-h_2-h_3}$\\
$h_2=0$ & $x^{h_0-h_1-h_3}$\\
$h_3=0$ & $(x-y)^{h_0-h_1-h_2}$\\ \hline
\end{tabular}
\end{table}
%
%
\begin{table}[h]
\caption{Virasoro conformal blocks}
\label{table_block2}
  \begin{tabular}{|l|c||c|} \hline
$(h_0,h_1,h_2,h_3)$ & $h$ & $C_{h_0,h_1,h_2,h_3}^{h}(x,y)$ \\ \hline \hline
$(\frac{1}{16},\frac{1}{16},\frac{1}{16},\frac{1}{16}) $
& $0$ & $\frac{1}{2}\{xy(x-y)\}^{-\frac{1}{8}}
\Bigl((x^{\frac{1}{2}}+y^{\frac{1}{2}})^{\frac{1}{2}} +
(x^{\frac{1}{2}}-y^{\frac{1}{2}})^{\frac{1}{2}} \Bigr)$
 \\
    & $\frac{1}{2}$ & $\frac{1}{2} \{xy(x-y)\}^{-\frac{1}{8}}
\Bigl((x^{\frac{1}{2}}+y^{\frac{1}{2}})^{\frac{1}{2}} - 
(x^{\frac{1}{2}}-y^{\frac{1}{2}})^{\frac{1}{2}} \Bigr)$
 \\ \hline
$(\frac{1}{2},\frac{1}{2},\frac{1}{2},\frac{1}{2}) $
& $0$ & $\{xy(x-y)\}^{-1}(x^2-xy+y^2)$ \\ \hline

$(\frac{1}{2},\frac{1}{2},\frac{1}{16},\frac{1}{16}) $
& $0$ & $y^{-\frac{1}{8}}\{x(x-y)\}^{-\frac{1}{2}}
(x-\frac{y}{2})$ \\
$(\frac{1}{2},\frac{1}{16},\frac{1}{2},\frac{1}{16}) $
& $\frac{1}{16}$ &
$\frac{1}{2}x^{-\frac{1}{8}}\{y(x-y)\}^{-\frac{1}{2}}
(x-2y)$ \\
$(\frac{1}{16},\frac{1}{2},\frac{1}{2},\frac{1}{16}) $
& $\frac{1}{16}$ & $\frac{1}{2}(xy)^{-\frac{1}{2}} (x-y)^{-1}(x+y)$\\
$(\frac{1}{2},\frac{1}{16},\frac{1}{16},\frac{1}{2}) $
& $\frac{1}{16}$ & $\frac{1}{2}\{xy\}^{-\frac{1}{2}}(x-y)^{-\frac{1}{8}}(x+y)$ \\
$(\frac{1}{16},\frac{1}{2},\frac{1}{16},\frac{1}{2}) $
& $\frac{1}{16}$ & $\frac{1}{2}x^{-1}\{y(x-y)\}^{-\frac{1}{2}}(x-2y)$\\
$(\frac{1}{16},\frac{1}{16},\frac{1}{2},\frac{1}{2}) $
& $0$ & $y^{-1}\{x(x-y)\}^{-\frac{1}{2}}(x-\frac{y}{2})$ \\ \hline
\end{tabular}
\end{table}

In both Table~\ref{table_vacuum} and Table~\ref{table_block2}, the entries define multivalued functions.
Hence it is necessary to explain how these correspond to the single-valued analytic functions 
\[
C_{h_0,h_1,h_2,h_3}^h(z_{13}, z_{23})
\] 
that are defined on \(U_{1(23)}\).

We illustrate this correspondence in detail for the case where 
\(h_0,h_1,h_2,h_3 = \fs\), and verify that the resulting associator coincides with that of the Tambara-Yamagami tensor category. 
Recall that 
\[
U_{1(23)} = \{(z_1,z_2,z_3)\in X_3(\C) \mid |z_{13}| > |z_{23}|,\; z_{13}/z_{23}\notin \R_-,\; z_{13}\notin \R_-\}.
\]

Formally, one can write
\begin{align}
\begin{split}
&\frac12\{x y(x-y)\}^{-1/8}\Bigl((\sqrt{x}+\sqrt{y})^{1/2} + (\sqrt{x}-\sqrt{y})^{1/2}\Bigr)\\
&= \frac12 x^{-1/8}\{(y/x)(1-y/x)\}^{-1/8}\Bigl((1+(y/x)^{1/2})^{1/2} + (1-(y/x)^{1/2})^{1/2}\Bigr),
\end{split}
\label{eq_xxxx}
\end{align}
and since
\((1 \pm (y/x)^{1/2})^{1/2} = \sum_{k \ge 0} \binom{1/2}{k} (\pm y/x)^{k/2}\), 
the expression in \eqref{eq_xxxx} admits a convergent series expansion of the form:
\[
y^{-1/8}\left(1 + \frac{y}{x}\C\left[\left[\frac{y}{x}\right]\right]\right)
\]
for \(|x| > |y|\).  Substituting \((x,y) = (z_{13}, z_{23})\) yields a single-valued analytic function on \(U_{1(23)}\), which is denoted by \(\bigl|_{|x|>|y|}\).
Then Theorem \ref{list_block2} asserts the following identity as a formal series:
\begin{align}
\frac12\{x y(x-y)\}^{-1/8}
\Bigl((\sqrt{x}+\sqrt{y})^{1/2} + (\sqrt{x}-\sqrt{y})^{1/2}\Bigr) \Bigl|_{|x|>|y|}
= \bra{\fs} I_{\fs 0}^{\fs}\left(\ket{\fs},x\right)\,I_{\fs\fs}^{0}\left(\ket{\fs},y\right)\,\ket{\fs}.
\label{eq_fsfs_1}
\end{align}

Next, we determine the associator.  
Let $x > y > 0$ be real numbers, and consider a path from the region $x > y$ to the region $y > x - y$, while keeping the condition $x,y\in\R$ and $x> y>0$ and $x - y > 0$ throughout. Therefore, the path in $X_3(\C)$ defined by the map
$(x,y) \mapsto (x,y,0) \in X_3(\mathbb{C})$
remains inside the intersection $U_{1(23)} \cap U_{(12)3}$, and in particular, does not cross any branch cuts.
In this region, the functions listed in Table~\ref{table_vacuum} and Table~\ref{table_block2} admit branches such that all $n$-th roots take positive real values.
This path represents the associator generator \(\alpha\) in $\CPaB(3)$. For example, in $|x|>|y|>0$, one has
\[
C_{\ft\fs\ft\fs}^\fs(x,y)
= \frac12 x^{-1/8} \{y(x-y)\}^{-1/2}(x-2y)\Bigl|_{|x|>|y|} \sim \ft\,x^{\ft-\fs-\fs} y^{\fs-\ft-\fs},
\]
which matches the behavior of \eqref{eq_CB_leading}.
On the other hand, expanding around \((x-y)/y \to 0\)
\begin{align*}
\frac12 x^{-1/8} \{y(x-y)\}^{-1/2}(x-2y)\Bigl|_{|y|>|x-y|}
&= \frac12(y+(x-y))^{-1/8}\{y(x-y)\}^{-1/2}(-y+(x-y))\Bigl|_{|y|>|x-y|}\\
&\sim -\frac12\,y^{\ft-\fs-\fs}(x-y)^{\fs-\ft-\fs}.
\end{align*}
holds. 
In this region, the latter expansion defines a section of \(\CB_{\ft\fs\ft\fs}(U_{(12)3})\).
Comparing the leading terms, we conclude:
\begin{align}
\frac12 x^{-1/8} \{y(x-y)\}^{-1/2}(x-2y) \Bigl|_{|y|>|x-y|}
= -\bra{\ft} I\left(I\left(\ket{\fs},x-y\right)\ket{\ft},y\right)\ket{\fs},
\label{eq_ass_non_trivial2}
\end{align}
where the minus sign coincides precisely with the nontrivial associator \(\alpha_{m,\ft,m} = \id_0 - \id_{\ft}\) in the Tambara-Yamagami tensor category.
Another nontrivial case arises for \(C_{\fs\fs\fs\fs}^h(x,y)\) with \(h\in\{0,\ft\}\).  Similarly one shows
\begin{align*}
\frac12\{x y(x-y)\}^{-1/8}\Bigl(\left(\sqrt{x}+\sqrt{y}\right)^{1/2} \pm \left(\sqrt{x}-\sqrt{y}\right)^{1/2}\Bigr)\Bigl|_{|x|>|y|}
= \bra{\fs} I_{\fs h_\pm}^{\fs}\left(\ket{\fs},x\right)\,I_{\fs\fs}^{h_\pm}\left(\ket{\fs},y\right)\,\ket{\fs},
\end{align*}
with \(h_+=0\) and \(h_-=\ft\).
Since
\begin{align*}
\ft\Bigl((\sqrt{x}+\sqrt{y})^{1/2} \pm (\sqrt{x}-\sqrt{y})^{1/2}\Bigr)
&= \ft\,y^{1/4}\Bigl(\left(2+\ft (x-y)/y - \cdots\right)^{1/2} \pm \left(\ft (x-y)/y - \cdots\right)^{1/2}\Bigr)\\
&\sim \frac{1}{\sqrt2}y^{1/4} \pm \frac{1}{2\sqrt2}y^{1/4-\ft}(x-y)^\ft,
\end{align*}
we deduce
\begin{align*}
\frac12\{x y(x-y)\}^{-1/8}\Bigl((\sqrt{x}+\sqrt{y})^{1/2} \pm (\sqrt{x}-\sqrt{y})^{1/2}\Bigr)\Bigl|_{|y|>|x-y|}
\sim \frac{1}{\sqrt2}(x-y)^{-\fs-\fs} \pm \frac{1}{2\sqrt2}y^{-\ft}(x-y)^{\ft-\fs-\fs}.
\end{align*}
Hence the connection matrix of
\[
  A(\alpha):
  \CB_{\fs\fs\fs\fs}(U_{(12)3})
  \longrightarrow
  \CB_{\fs\fs\fs\fs}(U_{1(23)})
\]
is $\frac{1}{\sqrt{2}}
  \begin{pmatrix}
    1 & 1 \\
    1 & -1
  \end{pmatrix}$.
The categorical associator is obtained from the inverse transpose of this connection matrix. Since this matrix is symmetric and involutive, the resulting associator is the same matrix, and hence agrees with \(\alpha_{m,m,m}\) in \eqref{eq_al_mmm}.
Finally, the behavior of Virasoro conformal blocks as \(x-y\to0\) is
summarized in the following table.
\begin{table}[h]
\caption{Asymptotic behavior of Virasoro conformal blocks}
\label{table_block3}
  \begin{tabular}{|l|c||c|c|} \hline
$(h_0,h_1,h_2,h_3)$ & $h$ & $C_{h_0,h_1,h_2,h_3}^{h}(x,y)$ & $x-y \to 0$ \\ \hline \hline
$(\frac{1}{2},\frac{1}{2},\frac{1}{2},\frac{1}{2}) $
& $0$ & $\{xy(x-y)\}^{-1}(x^2-xy+y^2)$ & $(x-y)^{-1}$ \\ \hline

$(\frac{1}{2},\frac{1}{2},\frac{1}{16},\frac{1}{16}) $
& $0$ & $y^{-\frac{1}{8}}\{x(x-y)\}^{-\frac{1}{2}}
(x-\frac{y}{2})$ & $\ft (x-y)^{\fs-\ft-\fs} y^{\ft-\fs-\fs}$
\\
$(\frac{1}{2},\frac{1}{16},\frac{1}{2},\frac{1}{16}) $
& $\frac{1}{16}$ &
$\frac{1}{2}x^{-\frac{1}{8}}\{y(x-y)\}^{-\frac{1}{2}}
(x-2y)$ & $- \frac{1}{2} y^{\ft-\fs-\fs}(x-y)^{\fs-\ft-\fs}$\\
$(\frac{1}{16},\frac{1}{2},\frac{1}{2},\frac{1}{16}) $
& $\frac{1}{16}$ & $\frac{1}{2}(xy)^{-\frac{1}{2}} (x-y)^{-1}(x+y)$ & $(x-y)^{-\ft-\ft}$\\
$(\frac{1}{2},\frac{1}{16},\frac{1}{16},\frac{1}{2}) $
& $\frac{1}{16}$ & $\frac{1}{2}\{xy\}^{-\frac{1}{2}}(x-y)^{-\frac{1}{8}}(x+y)$ & $(x-y)^{-\fs-\fs}$\\
$(\frac{1}{16},\frac{1}{2},\frac{1}{16},\frac{1}{2}) $
& $\frac{1}{16}$ & $\frac{1}{2}x^{-1}\{y(x-y)\}^{-\frac{1}{2}}(x-2y)$ & $-\ft (x-y)^{\fs-\ft-\fs} y^{\fs-\ft-\fs}$
\\
$(\frac{1}{16},\frac{1}{16},\frac{1}{2},\frac{1}{2}) $
& $0$ & $y^{-1}\{x(x-y)\}^{-\frac{1}{2}}(x-\frac{y}{2})$
& $\ft (x-y)^{\fs-\fs-\ft}y^{\fs-\fs-\ft}$ \\ \hline
\end{tabular}
\end{table}
Based on Table \ref{table_block3}, Table \ref{table_vacuum}, and the preceding discussion, we obtain the following result:
\begin{thm}
The module category of $L(\tfrac{1}{2},0)$ is equivalent, as a balanced braided tensor category, to 
$\TY(\Z_2,\chi,2^{-1/2})$.
\end{thm}

%

\section*{Acknowledgements}
This note was prepared as a contribution to the proceedings of the 2nd AMS-UMI International Joint Meeting. I would like to express my sincere gratitude to the organizers for their warm hospitality, and to the referees for their valuable comments.

\bibliographystyle{amsplain}
\bibliography{references}

\end{document}